\documentclass[a4paper,UKenglish,cleveref, autoref]{article}

\usepackage{enumerate}
\usepackage{textcomp}
\usepackage{array}
\usepackage{cite}
\usepackage{amsthm}
\usepackage{upgreek,xcolor,tikz,stackrel}
\usepackage[inline, shortlabels]{enumitem}
\usepackage{graphicx}
\usepackage[all]{xy}
\usepackage{tikz-cd}

\usetikzlibrary{decorations.pathmorphing}
\usetikzlibrary{arrows}
\usepackage{algorithm}

\usepackage{multicol}
\usepackage[shortlabels]{enumitem}
\newcommand{\quot}[2]{#1/#2}

\newcommand{\tdL}[1]{#1^\infty}

\newcommand{\quotpi}[1]{[#1]}

\newcommand{\ignore}[1]{}

\newcommand{\moment}{{\mathfrak m}}

\newcommand{\noment}{{\mathfrak n}}

\newcommand{\david}[1]{{\color{orange}\footnote{\color{orange}DAVID: #1}}}
\newcommand{\yoav}[1]{{\color{red}\footnote{\color{red}YOAV: #1}}}

\newcommand{\dn}{{\square}}
\newcommand{\dd}{{\sdiamond}}

\usepackage{thmtools}
\usepackage{amssymb}
\usepackage{amsfonts}
\usepackage{tikz}
\usepackage{apxproof}

\usepackage{float}
\usepackage{tikz-cd}
\usepackage{amsmath}

\usepackage{centernot}
\usepackage{stmaryrd}
\usepackage{stackengine}
\usepackage{amsmath,mathtools}
\usepackage{authblk}
\usepackage{wasysym}

\newcommand{\bc}{{\fullmoon}}

\newcommand{\lb}{\langle}
\newcommand{\rb}{\rangle}

\newcommand{\gog}{\mathfrak S}

\DeclareFontFamily{U}{MnSymbolC}{}
\DeclareSymbolFont{MnSyC}{U}{MnSymbolC}{m}{n}
\DeclareMathSymbol{\prediamondplus}{\mathbin}{MnSyC}{"7C}
\newcommand{\diamondplus}{{\prediamondplus}}
\DeclareMathSymbol{\prediamonddot}{\mathbin}{MnSyC}{"7E}
\newcommand{\diamonddot}{{\prediamonddot}}
\DeclareMathSymbol{\sdiamond}{\mathbin}{MnSyC}{"6E}

\DeclareFontShape{U}{MnSymbolC}{m}{n}{
    <-6>  MnSymbolC5
   <6-7>  MnSymbolC6
   <7-8>  MnSymbolC7
   <8-9>  MnSymbolC8
   <9-10> MnSymbolC9
  <10-12> MnSymbolC10
  <12->   MnSymbolC12}{}
\newcommand{\dtan}{{\dd^\infty}}
\newcommand{\ctan}{{\diamonddot}^\infty}

\usetikzlibrary{shapes.geometric,shapes.symbols}
\theoremstyle{plain}
\newtheorem{theorem}{Theorem}[section]
\newtheorem{lemma}[theorem]{Lemma}
\newtheorem{proposition}[theorem]{Proposition}
\newtheorem{corollary}[theorem]{Corollary}

\theoremstyle{definition}
\newtheorem{definition}[theorem]{Definition}
\newtheorem{example}[theorem]{Example}

\author[$\mathsection$]{David Fern\'andez-Duque}
\author[$\dagger$]{Yo\`av Montacute}

\affil[$\mathsection$]{Department of Mathematics, Ghent University} 
\affil[ ]{ICS of the Czech Academy of Sciences}
\affil[ ]{\texttt{david.fernandezduque@ugent.be}}
\affil[$\dagger$]{Computer Laboratory, University of Cambridge}
\affil[ ]{\texttt {yoav.montacute@cl.cam.ac.uk}}

\begin{document}

\title{Dynamic Tangled Derivative Logic of Metric Spaces}

\maketitle
 
\begin{abstract}
Dynamical systems are abstract models of interaction between space and time.
They are often used in fields such as physics and engineering to understand complex processes, but due to their general nature, they have found applications for studying computational processes, interaction in multi-agent systems, machine learning algorithms and other computer science related phenomena. 
In the vast majority of applications, a dynamical system consists of the action of a continuous `transition function' on a metric space.
In this work, we consider decidable formal systems for reasoning about such structures.

Spatial logics can be traced back to the 1940's, but our work follows a more dynamic turn that these logics have taken due to two recent developments: the study of the topological $\mu$-calculus, and the the integration of linear temporal logic with logics based on the Cantor derivative.
In this paper, we combine dynamic topological logics based on the Cantor derivative and the `next point in time' operators with an expressively complete fixed point operator to produce a combination of the topological $\mu$-calculus with linear temporal logic.
We show that the resulting logics are decidable and have a natural axiomatisation.
Moreover, we prove that these logics are complete for interpretations on the Cantor space, the rational numbers, and subspaces thereof.
\end{abstract}


\section{Introduction}
 
Our planet is orbited by a myriad of man-made satellites, whose movement must be predicted and controlled to e.g.~avoid collision with other objects.
To this end, their position over time is modelled using our knowledge of physics, and the mathematical structure governing this behaviour is known as a {\em dynamical system} (Figure \ref{fig:orbits1}).
Given the initial position and momentum of a satellite, one may predict the path it will take: it may be periodic, diverge into space or crash into the earth.
In such models, both space and time are continuous, i.e.~given by Euclidean spaces; however, they can be approximated discretely for a better computational treatment, or even be represented via finite relational structures (Examples \ref{centreofgravity} and \ref{metrictokripke}).
One can thus imagine satellites moving one `tick of the clock' at a time, for a suitably small time interval, a viewpoint that is generally better suited for our purposes.

Generally speaking, a discrete time dynamical system is defined to be a topological space equipped with a transition function, representing movement.
However, for applications regarding physical space, it is more convenient to work with metric spaces rather than arbitrary topological spaces.
For the logician, metric spaces offer some technical advantages, as the logics they provide are `better-behaved', but also present additional challenges, as extra care must be taken in ensuring that the structures produced are metrisable.
\begin{figure}\label{figOrbits}
    \centering
    \includegraphics[scale=0.6]{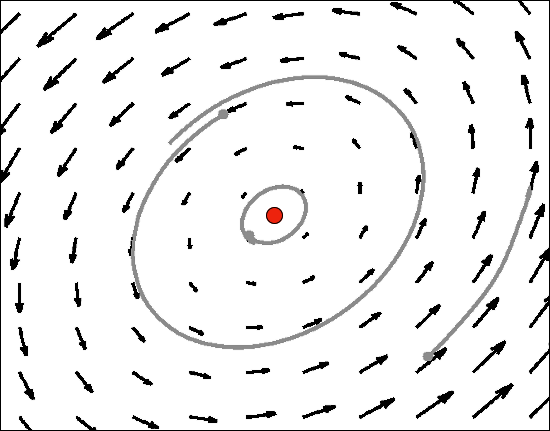}
    \caption{Orbits around a centre of gravity}
    \label{fig:orbits1}
\end{figure}
Two recent developments have taken spatial logic into a more `dynamic' direction.
The first is the development of the topological $\mu$-calculus \cite{BBFD21,Goldblatt2017Spatial}, which enriches the usual topological operators with fixed points, including Cantor's perfect core, and has applications in formal epistemology \cite{Surprise}.
The second, following a suggestion of Saveliev, is the combination of linear temporal logic with spatial logics based on the Cantor derivative~\cite{CSLpaper}, obtaining a more expressive version of Dynamic Topological Logic~\cite{artemov}.
This allows one to reason about e.g.~dense-in-themselves spaces, which are of relevance for example in chaos theory~\cite{devaney} and differential equations.
Our goal is to combine the expressive power of these two proposals and produce dynamic topological logics with topological fixed points.
We are specifically interested in dynamical systems based on metric spaces, as these are the spaces used in most applications.

Here we build on Fernández-Duque and Montacute~\cite{CSLpaper}, who consider a bi-modal language with $\dd$ interpreted as Cantor derivative and $\bc$ as `next point in time'.
In order to enrich this logic with topological fixed points, we follow Goldblatt and Hodkinson~\cite{Goldblatt2017Spatial}, who utilised results of Dawar and Otto~\cite{DawarO09} to represent the topological $\mu$-calculus via its relatively simple but expressively complete {\em tangled fragment}~\cite{FernandezTangled,FernandezSimulability}.
The latter augments modal logic with a polyadic modality $\dtan$, in which $\dtan\Gamma$ holds in the largest subspace where each $\varphi \in \Gamma$ is dense.
This grants us the full power of the topological $\mu$-calculus while working within a formal language that is relatively tame from a combinatorial perspective.
We thus obtain the logic ${\bf K4C}^\infty$ (and various other extensions) which plays the role of the standard dynamic topological logic $\bf S4C $.

Despite the additional expressive power due to the Cantor derivative and definability of topological fixed points, we show that ${\bf K4C}^\infty$ enjoys the same desirable properties of $\bf S4C $:
it is naturally axiomatisable and decidable over the class of all metric spaces.
Moreover, we extend a result of Mints and Zhang~\cite{MintsZ05} which states that $\bf S4C$ is complete for the Cantor space, by showing that ${\bf K4DC}^\infty$ (the extension of ${\bf K4C}^\infty$ with the `seriality' axiom) is sound and complete both for the Cantor set and for the set of rational numbers.
Aside from the above-mentioned logics, we also consider the logics  ${\bf K4I}^\infty$ and ${\bf K4DI}^\infty$ for dynamical systems where the transition function is an {\em immersion} (i.e.\ it preserves the Cantor derivative).
All of these logics are decidable, each logic with the $\rm D$ axiom is shown to be complete for the Cantor space and for the rational numbers, and logics without these axioms are shown to be complete for subspaces of these two metric spaces.

Working with the $\mu$-calculus is notoriously challenging, and despite the simplicity gained by working in the tangled fragment, there are still many non-trivial hurdles to overcome.
In order to deal with fixed points, we follow techniques pioneered by Fine~\cite{Fin74c} based on {\em final points}, already shown by Bezhanishvili et al.~\cite{BBFD21} to be useful for working with the topological $\mu$-calculus.
In our work, we further refine these techniques in order to deal with the interactions between the topology and the transition function.
Completeness for the Cantor space and the rational numbers is obtained via the technique of {\em dynamic p-morphisms}.
To apply it in our setting, we employ Kripke frames with limits, as used by Kremer and Mints~\cite{KremerM07}, along with the world-duplication construction from derivational modal logic (see e.g.~\cite{BBFD21}).
Our general method uniformly yields results for these two metric spaces and their closed subspaces. The current work is essential for the aixomatisation of $\mathbf{DTL}$ with the Cantor derivative, which currently has only been achieved in the setting of \emph{scattered spaces}; see Fern\'andez-Duque and Montacute\cite{Untangled}.

\subsection{Dynamical systems in computer science}
Dynamical systems are mathematical models of movement, routinely used in many pure and applied sciences, including computer science.
To cite some examples, in data-driven dynamical systems, many problems may be solved through techniques from dynamical systems as was recently suggested by Brunton and Kutz \cite{brunton2019data}. Dynamic theoretic approaches to machine learning also became very prominent in recent years. 
One instance of this was introduced by
Weinan~\cite{weinan2017proposal} and uses dynamical systems to model nonlinear functions employed in machine learning. 
Lin and Antsaklis's~\cite{hybrid} hybrid dynamical systems have been at the centre of research in control theory,  artificial intelligence and computer-aided verification. 
Dynamical systems are also present in the field of humonoid robotics, for instance in the study of movement learning via nonlinear dynamics used by Ijspeert et al.~\cite{robotics}. 
Mortveit and Reidys~\cite{sequential} introduced sequential dynamical systems which are discrete dynamical systems generalising notions such as cellular automata and  providing a framework through which one can study dynamical processes in graphs.
Another instance of dynamical systems in computer science can be found in linear dynamical systems, which are dynamical systems with  linear transformations (see Example~\ref{exCore}). 
Examples of such systems in computer science include Markov chains, linear recurrence sequences, and linear differential equations. 
Furthermore, there is a strong established relationship between dynamical systems and algorithms. 
This is present for example in the work of Hanrot, Pujol and Stehl\'e~\cite{hanrot2011analyzing}, and in the work of Chu~\cite{chu2008linear}.

\subsection{Outline.} In Section II we give the required definitions and notation necessary to understand the paper. 
In Section III we provide some background on the topic of dynamic topological logics. 
In section IV we share some motivation and applications for dynamic topological logic in dynamical systems, in particular in the context of metric spaces and computer science.
In Section V we introduce morphisms between models based on dynamical systems and provide some characterisation results for different instances of such models. 
In Section VI we present the canonical model for our base logics, and in Section VII we construct a finitary accessibility relation which will be used to prove completeness for different logics.
In Section VIII we introduce the notions of story and $\Phi$-morphism and use the previous results to prove the finite model property and Kripke $d$-completeness for different languages above $\mathbf{K4C}^\infty$.
In Section IX we prove topological $d$-completeness for several logics above $\mathbf{K4C}^\infty$ with respect to $\mathbb{Q}$ and the Cantor set.
Section X concludes with some final remarks.

\section{Preliminaries}\label{secPrel}
In this section we introduce the notation and definitions required for understanding this paper.
We work with the general setting of {\em derivative spaces}, in order to unify the metric space and Kripke semantics of our logics.

\begin{definition} [metric space]
A metric space is a pair $\mathfrak {X}=\langle X,\delta\rangle$, where $X$ is a set and $\delta\colon X\times X\to [0,\infty)$ is a {\em metric,} i.e. a map satisfying the following conditions for all $x,y,z\in X$:
\begin{itemize}
\item $\delta(x,y) = 0$ iff $x=y$

\item $\delta(x,y) = \delta(y,x)$

\item $\delta(x,y) + \delta (y,z) \leq \delta(y,z) $.

\end{itemize}
\end{definition}

The main operation on metric spaces that we are interested in is the {\em Cantor derivative.}

\begin{definition}[Cantor derivative] Let $\mathfrak {X}=\langle X,\delta \rangle$ be a metric space. Given $S\subseteq X$, the \emph{Cantor derivative} of $S$ is the set $d(S)$ of all limit points of $S$, i.e.\ \[x\in  d(S) \iff \forall \varepsilon > 0 \exists y \in X ( 0<\delta(x,y)<\varepsilon).\]
We may write $d(S)$ or $dS$ indistinctly.
\end{definition}

When working with more than one metric space, we may denote the Cantor derivative of the space $\langle X,\delta \rangle$ by $d_X $.
Given subsets $A,B\subseteq X$, it is not difficult to verify that $d $ satisfies the following properties:
\begin{enumerate}
\item\label{itDOne} $d(\varnothing)=\varnothing$;
\item\label{itDTwo} $d(A\cup B)= d(A)\cup d(B)$;
\item\label{itDThree} $dd(A) \subseteq  d(A)$.
\end{enumerate}
\noindent In fact, these conditions lead to the more general notion of {\em derivative spaces}:\footnote{Derivative spaces are a special case of {\em derivative algebras} introduced by Esakia \cite{EsakiaAlgebra}, where $\wp(X)$ is replaced by an arbitrary Boolean algebra.
We moreover work with `transitive' derivative algebras, so that the definition is stronger than that of e.g.~\cite{BBFD21}.
}

\begin{definition}
A {\em derivative space} is a pair $\mathfrak{X}=\lb X,\rho\rb$, where $X$ is a set and $\rho\colon \wp(X) \to \wp(X)$ is a map satisfying properties \ref{itDOne}-\ref{itDThree} above, where $d=\rho$.
\end{definition} 

Accordingly, if $\mathfrak X=\langle X,
\delta \rangle$ is a metric space and $d $ is the Cantor derivative on $\mathfrak X$, then $\lb X,d \rb$ is a derivative space.
However, there are other examples of derivative spaces.
The standard \emph{closure} of a subset $A$ of points in a topological space can be defined as $c(A) = A\cup d  (A)$.
Then, $\lb X,c\rb$ is also a derivative space, which satisfies the additional property $A\subseteq c(A)$; we call such derivative spaces {\em closure spaces.}
More generally, if $\lb X,\rho\rb$ is an arbitrary derivative space, we define $\dot\rho (A) := A\cup \rho(A)$; then, $\lb X,\dot\rho\rb$ is a closure space.

Another example of derivative spaces comes from transitive Kripke frames.
For the sake of succinctness, we call these frames {\em derivative frames.}
Below and throughout the text, we write $\exists x\sqsupset y \ \varphi $ instead of $\exists x (y\sqsubset x \wedge \varphi) $, and adopt a similar convention for the universal quantifier and other relational symbols.

\begin{definition}
A {\em derivative frame} is a pair $\mathfrak  F = \langle W,\sqsubset \rangle$ where $W$ is a non-empty set and $\sqsubset $ is a transitive relation on $W$.
We denote the reflexive closure of $\sqsubset $ by $\sqsubseteq$.
\end{definition}

We chose the notation $\sqsubset $ because it is suggestive of a transitive relation, but remains ambiguous regarding reflexivity, as there may be irreflexive and reflexive points.
We also write $w\equiv v$ if $w\sqsubseteq v$ and $v\sqsubseteq w$; the equivalence class of $w$ under $\equiv$ is called the {\em cluster} of $w$ and is denoted $C(w)$.

Given $A\subseteq W$, we define ${\downarrow_\sqsubset}$ as a map
${\downarrow_\sqsubset} \colon \wp(W) \to \wp(W)$ such that
$${\downarrow_\sqsubset}(A)= \{w\in W: \exists v\sqsupset w (v\in A)\}.$$
The following is then readily verified.

\begin{lemma}
If $\langle W,\sqsubset \rangle$ is a derivative frame, then $\langle W,{\downarrow_\sqsubset} \rangle$ is a derivative space.
\end{lemma}

Dynamical derivative systems consist of a derivative space equipped with a {\em continuous} function.
Recall that if $\langle X,\delta_X\rangle$ and $\langle Y,\delta_X\rangle$ are metric spaces and $f\colon X\to Y$, then $f$ is {\em continuous} if for every $x\in X$ and every $\varepsilon>0$ there exists $\eta>0$ such that $\delta_X(x,x')<\eta $ implies $\delta_Y(f(x),f(x'))<\varepsilon$.
It is well known (and not difficult to check) that $f$ is continuous iff $c_X  f^{-1}(A) \subseteq f^{-1}c_Y(A)$ for all $A\subseteq Y$.
We thus arrive at the following general definition.

\begin{definition}\label{defCH}
Let $\langle X,\rho_X \rangle$ and $\langle Y,\rho_Y\rangle $ be derivative spaces.
We say that $f\colon X\to Y$ is {\em continuous} if for all $A\subseteq Y$, $ \dot \rho_X f^{-1}(A) \subseteq     f^{-1}\dot \rho_Y (A)$.
We say that $f$ is an {\em immersion}\footnote{Normally immersions are defined to be locally injective, continuous maps. Our definition is a bit more general, but the actual immersions we will construct later are, indeed, locally injective.} if it satisfies the stronger condition $ \rho_X f^{-1}(A) \subseteq   f^{-1}  \rho_Y (A)$.
Finally, $f$ is a {\em homeomorphism} if it is a bijection satisfying $ \rho_X f^{-1}(A) =  f^{-1}  \rho_Y (A)$.
\end{definition}

For the most part we will focus on continuous functions and immersions, but homeomorphisms are worth mentioning, since this is the appropriate notion of isomorphism for derivative spaces.
We are particularly interested in the case where $X=Y$, which leads to the notion of {\em dynamic derivative system.}

\begin{definition}
A {\em dynamic derivative system} is a triple $\mathfrak S =\lb X,\rho,f\rb$, where $\lb X,\rho\rb$ is a derivative space and $f\colon X\to X$ is a continuous map.
\end{definition}

If $\mathfrak S = \langle X,\rho,f\rangle$ is such that $\rho $ is the Cantor derivative associated with a metric $\delta$, we say that $\mathfrak S$ is a {\em dynamic metric system} and identify it with the triple $\langle X,\delta ,f \rangle $.
If $\rho={\downarrow_\sqsubset }$ for some transitive relation $\sqsubset$, we say that $\mathfrak S$ is a {\em dynamic Kripke frame} and identify it with the triple $\langle X,\sqsubset ,f\rangle$.

It will be convenient to characterise dynamic Kripke frames in terms of the relation $\sqsubset $.

\begin{definition}[monotonicity and strict monotonicity]
\sloppy Let $\langle W,\sqsubset \rangle$ be a derivative frame.
A function $f\colon W\to W$ is {\em  monotonic} if $w\sqsubseteq  v$ implies $f(w)\sqsubseteq f(v)$, and {\em strictly monotonic} if $w\sqsubset  v$ implies $f(w)\sqsubset  f(v)$.
\end{definition}

\begin{lemma}
If $\langle W,\sqsubset \rangle$ is a derivative frame and $f\colon W\to W$, then
\begin{enumerate}

\item if $f$ is monotonic, then it is continuous with respect to $\downarrow_\sqsubset $, and

\item if $f$ is strictly monotonic, then it is an immersion with respect to $\downarrow_\sqsubset $.

\end{enumerate}
\end{lemma}

Next we will discuss the \emph{tangle operators}, which are important in spatial modal logic, as they are expressively equivalent to the $\mu$-calculus over the class of transitive Kripke frames, as shown by Dawar and Otto~\cite{DawarO09}.
In the topological context, the {\em tangled closure} was introduced by Fernández-Duque~\cite{FernandezTangled} and the {\em tangled derivative} was introduced by Goldblatt and Hodkinson \cite{Goldblatt2017Spatial}, who observed that Dawar and Otto's result holds for metric spaces as well.

\begin{definition}[tangled derivative]
Let $\lb X,\rho\rb$ be a derivative space and let $\mathcal{S}\subseteq \wp(X)$. Given $A\subseteq X$, we say that $\mathcal{S}$ is \emph{tangled in} $A$ if for all $S\in \mathcal{S}$,
$ A\subseteq \rho(S\cap A) $.
We define the \emph{tangled derivative} of $\mathcal{S}$ as 
$$ \rho^\infty (\mathcal{S}):=\bigcup \{ A\subseteq X : \mathcal{S}\text{ is tangled in }A\}.$$ 
\end{definition}

The {\em tangled closure} is then the special case of the tangled derivative where $\rho$ is a closure operator, and we denote it by $\dot\rho^\infty$ (or $c^\infty$ when working with a metric space).

\begin{example}
Let $\langle W,\sqsubset\rangle$ be a derivative frame, and $S_1,\ldots,S_n\subseteq W$.
Then, $w\in {\downarrow}_\sqsubset^\infty ( \{S_1,\ldots,S_n \})$ if and only if there is an infinite sequence
\[w= w_0\sqsubset w_1\sqsubset \ldots \]
such that for every $k=1,\ldots,n$, $w_i\in S_k$ for infinitely many values of $i$~\cite{Goldblatt2017Spatial}.

The case where $W$ is finite is particularly transparent.
In this case, the sequence $w_0\sqsubset w_1\sqsubset \ldots $ will eventually stabilise in a single  cluster; that is, for some $j$ we will have that $w_{i+1}\sqsubset w_i$ whenever $i>m$.
By transitivity, all such $w_i$ must be reflexive, so we arrive at the following characterisation:
$w\in \rho^\infty ( \{S_1,\ldots,S_n \})$ if and only if there is a reflexive cluster $C(v)$ (i.e., a cluster for which all of its points are reflexive) with $w\sqsubset v$ such that for all $k=1,\ldots ,n $, $S_k\cap C(v) \neq \varnothing$.
\end{example}

\begin{example}
Let $A=\mathbb Q$ be the set of rational numbers, and $B=\mathbb R\setminus\mathbb Q$ be its complement.
Then, $d^\infty(\{A,B\}) = \mathbb R$, which is readily checked since both $A,B$ are dense, hence $\{A,B\}$ is tangled in $\mathbb R$.

If instead we define $A=(-\infty,0]$ and $B=[0,\infty)$, we get that $d^\infty(\{A,B\}) = \varnothing $.
This is because if $D\subseteq \mathbb R$, we cannot have that $\{A,B\}$ is tangled in $D$: if $D$ contains a negative number, then $D\not\subseteq d(B\cap D)$; if it contains a positive number, then $D\not \subseteq d(A\cap D)$.
So we are left with the case where $D=\{0\}$.
But then $d(D)=\varnothing$, so also $d(A\cap D)=\varnothing$.
In contrast, observe that in this case, we have $c^\infty(\{A,B\}) = \{0\} $, since $c(\{0\}) = \{0\}$.
\end{example}

Our goal is to reason about various classes of dynamic derivative systems	using the logical framework defined in the next section.

\section{Dynamic topological logics}\label{secDTL}

In this section we discuss dynamic topological logic in the general setting of dynamic derivative systems.
\sloppy Given a non-empty set $\mathsf{PV}$ of propositional variables, the language $\mathcal{L}_{\dd\dtan}^{\circ}$ is defined recursively as follows:
$$
\varphi::= p \; | \; \varphi\wedge \varphi \; | \; \neg\varphi \; | \; {\dd}\varphi \; | \;   \dtan \Phi  \; | \; \bc\varphi,$$
where $p \in \mathsf{PV}$ and $\Phi$ is a finite sequence of formulas in $\mathcal{L}_{\dd\dtan}^{\circ}$.
It consists of the Boolean connectives $\wedge$ and $\neg$, the temporal modality $\bc$, the modality $\dd$ for the derivative operator, and the tangled derivative modality $\dtan$. 
As usual, $\square := \neg\dd\neg$ is the dual of $\dd$.
The closure and interior modalities may be defined by $\diamonddot\varphi:=\varphi\vee\dd\varphi$ and $\boxdot\varphi:=\varphi\wedge\square\varphi$.
Following~\cite{Goldblatt2017Spatial}, we define $\ctan\Phi:=\diamonddot\bigwedge\Phi\vee \dtan\Phi$. 
\begin{definition}[semantics]\label{d-semantics}
A \emph{dynamic derivative model} (DDM) is a quadruple $\mathfrak {M}=\langle X,\rho,f,\nu\rb$ where $\langle X,\rho,f\rb$ is a dynamic derivative system and $\nu:\mathsf{PV}\rightarrow\wp(X)$ is a valuation function assigning a subset of $X$ to each propositional letter in $\mathsf{PV}$.  
Given $\varphi,\psi \in\mathcal{L}_{\dd\dtan}^{\circ }$, we define the truth set $\|\varphi\| \subseteq X$ of $\varphi$ inductively as follows:
\begin{itemize}

	\item $\| p \| = \nu(p)$;
	\item $\| \neg \varphi \| = X \setminus \|  \varphi \| $;
	\item $\| \varphi\wedge\psi \| = \|\varphi\| \cap \|\psi\| $;
	\item $\|\dd\varphi\| = \rho ( \|\varphi\| ) $;
			\item $\|\dtan \{\varphi_1,\ldots,\varphi_n\}\| = \rho^\infty(\{\|\varphi_1\|,\ldots,\|\varphi_n\|\}) $;
		\item $\| \bc\varphi \| = f^{-1}(\| \varphi\|)$.

\end{itemize}
We write $\mathfrak  M,x\models\varphi$ if $x\in \|\varphi\|$, and $\mathfrak  M\models\varphi$ if $\|\varphi\| = X$.
We may write $\|\cdot\|_\mathfrak M$ or $\|\cdot\|_\nu$ instead of $\|\cdot\|$ when working with more than one model or valuation.
\end{definition}

The notion of {\em validity} is defined as usual; if $\mathfrak X$ is a dynamic derivative system and $\varphi$ is a formula, we write $\mathfrak X\models\varphi$ if $\langle\mathfrak X,\nu\rangle \models\varphi$ for every valuation $\nu$ on $\mathfrak X$.
Similarly, if $\Omega$ is a class of dynamical systems or models, we write $\Omega\models\varphi$ and say $\varphi$ is {\em valid on $\Omega$} if $\mathfrak A\models\varphi$ for every $\mathfrak A\in\Omega$.

We define other connectives (e.g.\ $\vee,\rightarrow$) as abbreviations in the usual way.
The fragment of $\mathcal{L}_{\dd}^{\circ }$ that includes only $\dd$ will be denoted by $\mathcal L_{\dd}$.
In order to align with the familiar axioms of modal logic, it is convenient to discuss the semantics of $\dn$.
Accordingly, we define the dual of the derivative, called the \emph{co-derivative}.

\begin{definition}[co-derivative]
Let $\langle X,\rho\rangle$ be a derivative space.
For each $S\subseteq X$ we define $\hat \rho(S):=X\backslash  \rho(X\backslash S)$ to be the {\em co-derivative} of $S$.
\end{definition}
The co-derivative satisfies the following properties, where $A,B\subseteq X$:
\begin{enumerate}
\item $\hat \rho (X)=X$;
\item $A\cap \hat \rho(A)\subseteq \hat \rho \hat \rho(A)$;
\item $\hat \rho(A\cap B)=\hat \rho(A)\cap \hat \rho(B)$.
\end{enumerate}

It can readily be checked that for every dynamic derivative model $\langle X,\rho,f,\nu\rangle$ and every formula $\varphi$, $\|\dn\varphi\| = \hat \rho (\|\varphi\|)$.
The co-derivative can be used to define the standard {\em interior} of a set, given by $i(A)=A\cap\hat \rho(A)$ for each $A\subseteq X$. This implies that $U\subseteq \hat \rho(U)$ for each open set $U$, but not necessarily $\hat \rho(U)\subseteq U$.
Next, we discuss the systems of axioms that are of interest to us.

Let us list the axiom schemes and rules that we will consider in this paper.
Below, if $\Phi = \{\varphi_1,\ldots,\varphi_n\}$ is a set of formulas then $\bc\Phi:=\{\bc\varphi_1,\ldots,\bc\varphi_n\}$, and $\diamondplus\in \{\dd,\diamonddot\}$.
\begin{description}
\item{\rm Taut} $:= \text{All propositional tautologies}$
\item{\rm K} $:= \dn(\varphi\to\psi)\to(\dn\varphi\to \dn\psi)$

\item{\rm 4} $  :=    \dn\varphi \to\dn\dn\varphi $

\item{\rm D} $  :=    \dd\top $
\item{${\rm Next}_\neg$} $:=\neg\bc\varphi\leftrightarrow\bc\neg\varphi$
\item{${\rm Next}_\wedge$} $:=\bc (\varphi\wedge\psi)\leftrightarrow \bc \varphi \wedge\bc \psi $
\item{${\rm C}_\diamondplus$} $:= \diamondplus \bc  \varphi\to  \bc\diamondplus\varphi$

\item{\rm MP} $:= \dfrac{\varphi \ \ \varphi\to \psi}\psi$
\item{${\rm Nec}_\dn$} $:= \dfrac{\varphi }{\dn \varphi}$
\item{${\rm Nec}_\bc$} $:= \dfrac{\varphi }{\bc \varphi}$

\item{${\rm Fix}_{\dd^\infty}$} $:= \dd^\infty \Phi \to \bigwedge_{\varphi\in \Phi} \dd (\varphi \wedge \dd^\infty \Phi)$

\item{${\rm Ind}_{\dd^\infty}$} $:= \boxdot \big (\theta \to  \bigwedge_{\varphi\in \Phi} \dd (\varphi \wedge \theta ) \big ) \to ( \theta\to  \dd^\infty \Phi ) $

\item {${\rm CTan}_{\diamondplus}$}$:= \diamondplus^\infty \bc\Phi \to \bc \diamondplus^\infty \Phi $

\end{description}
The `base modal logic' over $\mathcal L_\dd$ is given by
$$\mathbf{K}:= {\rm Taut}+{\rm K} +{\rm MP}+{\rm Nec}_\dn.$$
However, we are mostly interested in proper extensions of $\mathbf K$.
Let $\Lambda,\Lambda'$ be logics over languages $\mathcal L$ and ${\mathcal L}'$. 
We say that $\Lambda$ extends $\Lambda'$ if $\mathcal L'\subseteq\mathcal L$ and all the axioms and rules of $\Lambda'$ are derivable in $\Lambda$.
A logic over $\mathcal L_\dd$ is {\em normal} if it extends $\mathbf{K}$.
If $\Lambda$ is a logic and $\varphi$ is a formula, we denote by $\Lambda+\varphi$ the least extension of $\Lambda$ which contains every substitution instance of $\varphi$ as an axiom.

We write $\vdash_\Lambda\varphi$ when $\varphi$ is a theorem of $\Lambda$, or simply $\vdash\varphi$ when $\Lambda$ is clear from context.
Recall that a logic $\Lambda$ is {\em sound} for $\Omega$ if every theorem of $\Lambda$ is valid on $\Omega$, and {\em complete} if whenever $\Omega\models\varphi$, it follows that $\vdash_\Lambda\varphi$.

We then define
$\mathbf{K4}:=  \mathbf{K}+{\rm 4}$, $\mathbf{K4D} = \mathbf{K4} +\rm D $, and $\mathbf{S4}:=  \mathbf{K4}+{\rm T}$.
These logics are well known and characterise certain classes of spaces and Kripke frames which we review below.
For a logic $\Lambda$, $\tdL\Lambda$ denotes the logic $\Lambda + {\rm Fix}_{\dtan} + {\rm Ind}_{\dtan} $ over $\mathcal L _{\dd\dtan}$.

\begin{lemma}\label{lemmTanProp}
Let $\Gamma,\Delta$ be sets of formulas and let $\varphi$ be a formula.
\begin{enumerate}

\item If $\Gamma\subseteq\Delta$, then $\mathbf{K4}^\infty \vdash\dtan\Delta\to\dtan\Gamma$.

\item $\mathbf{K4}^\infty \vdash \ctan(\Gamma\cup\{\varphi,\neg\varphi\}) \to \dtan\Gamma$.

\end{enumerate}
\end{lemma}
\begin{proof}
The first item follows from using ${\rm Fix}_{\dd^\infty}$ to observe that $\dtan\Delta$ satisfies the premise of ${\rm Ind}_{\dd^\infty}$ applied to $\Gamma$.
The second follows from the definition $\ctan\Gamma:=\diamonddot \left( \bigwedge \Gamma\wedge \varphi \wedge \neg\varphi \right)  \vee \dtan(\Gamma \cup \{\varphi,\neg\varphi\})$; since  $\varphi \wedge \neg\varphi$ is inconsistent, this is equivalent to $\dtan(\Gamma \cup \{\varphi,\neg\varphi\})$, which by the first item implies $\dtan\Gamma$.
\end{proof}

In addition, for a logic $\Lambda$ over $\mathcal L_\dd$, $\Lambda\mathbf{F}$ is the logic over $\mathcal L^\circ_\dd$ given by
$$\Lambda\mathbf{F} := \Lambda+{\rm Next}_\neg+{\rm Next}_\wedge+{\rm Nec}_\bc.\footnote{Logics of the form $\Lambda\mathbf{ F}$ correspond to dynamical systems with a possibly discontinuous function. We will not discuss discontinuous systems in this paper; see \cite{artemov} for more information.}$$
This simply adds axioms of linear temporal logic to $\Lambda$, which hold whenever $\bc$ is interpreted using a function.

For continuous functions, we define
\[
\begin{array}{rclcrcl}
\Lambda{\bf C}& := &\Lambda\mathbf{F}+ {\rm C}_{\diamonddot} & \ \
  & \Lambda{\bf C}^\infty& :=& \Lambda^\infty\mathbf{C}+  {\rm CTan}_\diamonddot \\
 \Lambda{\bf I}&:=&\Lambda\mathbf{F}+  {\rm C}_{\Diamond}  
 &   & \Lambda{\bf I}^\infty&:=&\Lambda^\infty \mathbf{I} +   {\rm CTan}_\Diamond .
\end{array}
\]
As we will see, these correspond to derivative spaces with a continuous function or immersion, respectively; accordingly, we say that logics that include ${\rm C}_\Diamond$ are {\em immersive.}
The following is well known and dates back to McKinsey and Tarski \cite{Tarski}.

\begin{theorem}\label{theoK4metric}
$\mathbf{S4}$ is the logic of all closure spaces, the logic of all transitive, reflexive derivative frames, and the logic of the real line with the standard closure.
\end{theorem}

It is well known that $\mathbf{K4}$ is the logic of transitive derivative frames (see e.g.~\cite{black}), and Bezhanishvili and Lucero-Bryan~\cite{BL-B} showed it to be the logic of all countable metric spaces.

\begin{theorem}
$\mathbf{K4}$ is the logic of all (finite) derivative frames (i.e., transitive Kripke frames) and of all (countable) metric spaces.
\end{theorem}

Logics with the $\rm C$ axioms correspond to classes of dynamical systems.

\begin{lemma}\label{lemCH}
If $\Lambda$ is sound for a class of derivative spaces $\Omega$, then:
\begin{enumerate}

\item $\Lambda{\bf C}$ is sound for the class of dynamic derivative systems $\langle X,\rho,f\rangle$, where $\langle X,\rho\rangle\in \Omega$ and $f$ is continuous.

\item $\Lambda{\bf I}$ is sound for the class of dynamic derivative systems $\langle X,\rho,f\rangle$, where $\langle X,\rho\rangle\in \Omega$ and $f$ is an immersion.

\end{enumerate}
\end{lemma}

The above lemma is easy to verify from the definition of a continuous function in the context of derivative spaces (Definition \ref{defCH}).
Note that $\dtan \bc\Phi\to \bc\dtan\Phi$ is {\em not} valid over the class of dynamic derivative spaces with a continuous function (see Example~\ref{exCore}).

\section{Applications to dynamical systems}

Our logical framework is designed for the specification and formal reasoning about dynamical systems, especially those based on metric spaces.
In many applications, the spaces used have the additional property that they are {\em crowded}, or {\em dense-in-themselves} i.e., they have no isolated points.
In $d$-semantics, this property is expressed by the axiom $\rm D$, i.e.~$\dd\top$.

In the introduction, we mentioned an example involving satellites orbiting a centre of gravity.
Let us revisit this example with our formal language in mind.

\begin{example}[centre of gravity]\label{centreofgravity}
In Figure \ref{figOrbits}, we illustrate a model of bodies orbiting a centre of gravity on a plane.
We may model this as $\mathbb R^2$ with a transition function $f:\mathbb R^2\to \mathbb R^2$ corresponding to the movement of a body over a fixed time interval of $\varepsilon$ seconds.
We may then describe various properties of this system using dynamic topological logic.

First, observe the region $P$.\footnote{More accurately, this region should be denoted $\nu(P)$, but we will simply write $P$ for the sake of illustrations.}
Points in this region will return to $P$ after completing a full orbit (say, in time $n$), but not before that.
This corresponds to the expression $  P\to  \bc^n  P\wedge \neg \bigvee_{i=1}^{n-1} \bc^i  P$.
Conversely, the region $Q$ is a unsafe zone which none of the three orbits indicated in the figure intersect.
Accordingly, $\boxdot P\rightarrow \bigwedge _{i=0}^m \bc^i \boxdot \neg Q$ holds in our model for every $m$;
note that $\boxdot \neg Q$ means that we are inside the region $\neg Q$, not on the boundary.
This is important in a spatial safety condition, since it means that we are guaranteed \emph{not} to be in the unsafe region even if there is a small error in measurement.

\begin{figure}[H]
    \centering
    \includegraphics[scale=0.6]{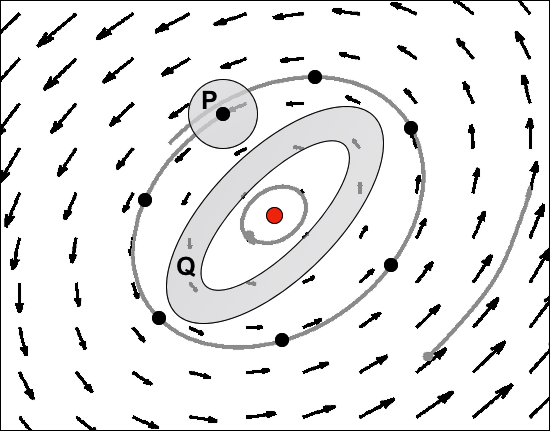}
    \caption{Orbits around a centre of gravity}
    \label{fig:orbits2}
\end{figure}
\end{example}
This is a basic example of a dynamical system arising from a metric space which is influenced by a force, in this case gravity.
Such forces can initiate different phenomena such as \emph{chaos} in the system.

Given a dynamical system $\mathfrak X=\langle X,\tau,f\rangle$, we say that $f:X\rightarrow X$ is \emph{topologically transitive} if for every nonempty open sets $U,V\in \tau$ there exists $n\geq 0$ such that $f^n(U)\cap V\neq \varnothing$. 
This is an important property that together with the set of periodic points of $f$ being dense implies that $\mathcal X$ is a \emph{chaotic} dynamical system;\footnote{There are many alternative definitions to mathematical chaos. We are referring to the original definition by Devaney \cite{devaney}.} in a seminal result, Banks et al.~\cite{banks} showed that such systems exhibit sensitive dependence on initial conditions, i.e.\ the `butterfly effect'.
\begin{example}

    Consider the dynamical system in Figure \ref{toptrans}. 
    Suppose that starting at each of the black points the function reaches the area $P$ within $n$ steps. 
    Then the formula $\diamonddot \bigvee_{i=1}^n \bc^i P$ captures the fact that each open neighbourhood of the red point contains a point reaching $P$ after some amount of time bounded by $n$.
    The existence of such $n$ is guaranteed by topological transitivity.

  \begin{figure}[H]
      \centering
      \includegraphics[scale=0.3]{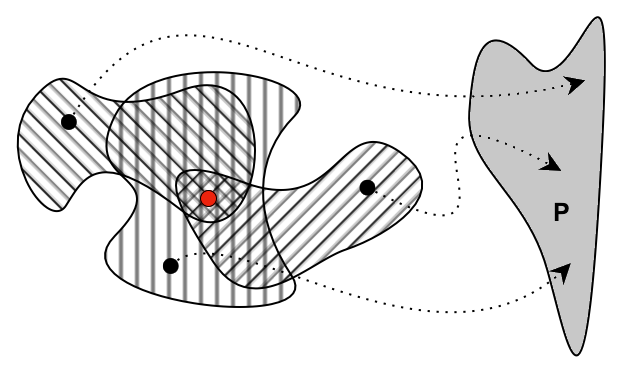}
      \caption{In a system exhibiting topological transitivity, the orbit of each open set intersects every other open set.}
      \label{toptrans}
  \end{figure}
\end{example}

Let us now turn our attention to \emph{topological fixed points}.
Recall that the $\mu$-calculus enriches modal logic with expressions of the form $\mu x.\varphi(x)$, where $x$ appears in the scope of an even number of negations in $\varphi$.
We denote the language of the $\mu$-calculus by $\mathcal L_\mu$.
The intended meaning of this expression is the least fixed point of the map $A \mapsto \varphi(A)$, where $A$ ranges over the subsets of some model $X$.
This notion makes sense when $X$ is a topological space or a metric space.
Dawar and Otto~\cite{DawarO09} showed that the bisimulation-invariant fragment of monadic second order logic ($\bf MSO$) is expressively equivalent to $\mathcal L_{\dd\dtan}$ over the class of finite $\bf K4$ frames. 
Since the $\mu$-calculus is a bisimulation-invariant fragment of $\bf MSO$, Goldblatt and Hodkinson~\cite{Goldblatt2017Spatial} observed that as a corollary, we obtain that for every $\varphi\in\mathcal L_\mu$, there is $\varphi^\infty \in \mathcal L_{\dd\dtan}$ such that $\varphi\leftrightarrow\varphi^\infty$ is valid over the class of metric spaces.
Thus no generality is lost when replacing $\mathcal L_\mu$ with $\mathcal L_{\dd\dtan}$.
When enriched with $\bc$, we obtain a logic where all {\em topological} fixed points can be expressed, but not those defined in terms of $\bc$, such as the `until' operator.

As an important special case, we consider the unary tangle $\dtan \{P\}$ which represents the {\em perfect core} of $P$, i.e.~the largest subset of $\nu( P)$ without isolated points.

\begin{example}\label{exCore} 
In figure \ref{fig:perfectcore}, we see two dynamical systems based on linear transformations on the plane: on the left a rotation, and on the right, a trivial system that maps the entire plane to $0$.
	The system on the left is an immersion (in fact, a {\em homeomorphism}), but the one on the right is not.
	Let $P$ be the top square on the left-hand figure (including both the interior and the boundary), and let $Q$ be the bottom square.
	It should be clear that $\| {\dtan \{P\}}\| = P$, since $P = d(P) = d(P\cap P) $.
	In other words, $P$ is {\em perfect,} i.e.~it is closed and contains no isolated points.
	Similarly, points in $Q$ satisfy $  \dtan \{ \bc P\}$, since every point of $Q$ satisfies $\bc P$ and $Q$ is also perfect.
	Moreover, these points also satisfy $\bc \dtan\{P\}$, so $  \dtan \{ \bc P\} \to \bc \dtan\{P\}$ holds; this is an instance of the axiom ${\rm CTan}_\dd$.
	
	In contrast, let us consider the figure on the right, and let $O $ be the singleton containing the origin; $Q$ is as above.
	As before, we have that every point of $Q$ satisfies $\bc O$, hence since $Q$ is perfect, then $\dtan\{\bc O\}$.
	However, the origin is an isolated point, i.e.~not perfect, so it does not satisfy $\dtan \{O\}$.
	It follows that points of $Q$ satisfy $\dtan\{\bc O\} \wedge \neg \bc \dtan \{O\} $, and ${\rm CTan}_\dd$ fails.
	However, the map is still continuous, so we expect ${\rm CTan}_\diamonddot$ to hold; and, indeed, we observe that $\ctan \{O\} $ holds on the origin, since $O\subseteq c(O)$.
	It readily follows that $\ctan\{\bc O\} \to \bc \ctan \{O\} $ is true in the model on the right, i.e.~${\rm CTan}_\diamonddot$ is valid.
	
	\begin{figure}[H]
	\centering
	\includegraphics[scale=0.25]{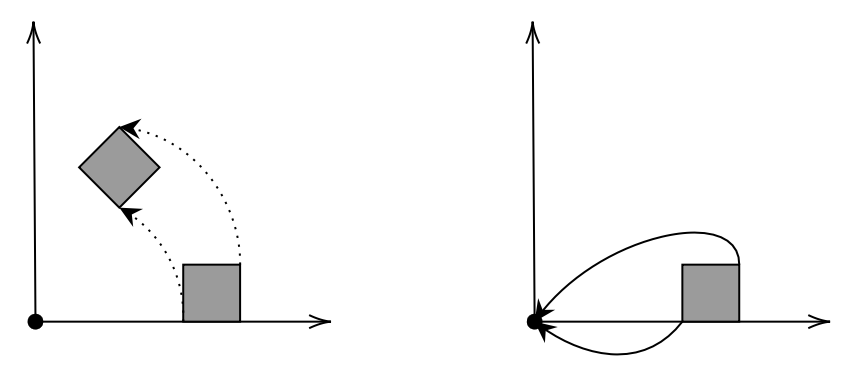}
	\caption{Two dynamical systems on the plane.}
	\label{fig:perfectcore}
\end{figure}
\end{example}

\section{Morphisms between dynamical systems}

In the study of modal logic, it is often useful to work with morphisms between structures preserving validity of formulas.
For Kripke semantics, such maps are called {\em $p$-morphisms.}
These morphsims can be defined and generalised in the context of dynamic derivative spaces as follows.

\begin{definition}[dynamic $p$-morphism]\label{defDmor}
Let $\mathfrak X=\langle X, \rho _\mathfrak X, f_\mathfrak X  \rangle$ and $\mathfrak Y=\langle Y,   \rho_\mathfrak Y, f_\mathfrak Y  \rangle$ be dynamic derivative systems. 
Let $\pi\colon X\to Y$.
We say that $\pi$ is a {\em dynamic $p$-morphism} if
\begin{itemize}
    \item $\pi^{-1}\rho_{\mathfrak Y}(B)=\rho_{\mathfrak X}\pi^{-1}(B)$ for all $B\subseteq Y$, and
    \item $\pi\circ f_{\mathfrak X}=f_{\mathfrak Y}\circ\pi$.
\end{itemize}
\end{definition}
\sloppy
These maps preserve validity in the following sense.

\begin{proposition}\label{propDmap}
Let $\mathfrak X=\langle X, \rho _\mathfrak X, f_\mathfrak X  \rangle$ and $\mathfrak Y=\langle Y,   \rho_\mathfrak Y, f_\mathfrak Y  \rangle$ be dynamic derivative spaces, and suppose that $\pi\colon X\to Y$ is a dynamic $p$-morphism.
Then, for every $\varphi\in\mathcal L_{\dd\dtan}$, if $\mathfrak X\models\varphi$ then $\mathfrak Y\models\varphi$.
\end{proposition}

The proof is standard; see e.g.~\cite{CSLpaper,Goldblatt2017Spatial}.
In practice, we will never use Definition~\ref{defDmor} in this general way; instead, we will specialise it to the specific classes of derivative spaces we are interested in.
Let us begin by describing dynamic $p$-morphisms between derivative frames; see~\cite{BBFD21} for details.

\begin{proposition}
Let $\mathfrak M=\langle  M, \sqsubset   _\mathfrak M, g_\mathfrak M  \rangle$ and $\mathfrak N=\langle N, \sqsubset   _\mathfrak N, g_\mathfrak N  \rangle$ be dynamic derivative frames.
Let $\pi\colon M \to N$.
Then $\pi$ is a dynamic $p$-morphism if
\begin{itemize}
    \item (forth condition) $w\sqsubset  _\mathfrak M v$ implies that $\pi(w) \sqsubset  _\mathfrak N\pi(v)$;
    \item  (back condition) $\pi(w)\sqsubset  _\mathfrak N u$ implies that there is $v\sqsupset_\mathfrak M w$ with $\pi(v) = u$;
    \item $\pi\circ g_\mathfrak M =g_\mathfrak N\circ \pi$.
\end{itemize}
\end{proposition}

Our main results are obtained by first establishing them for derivative frames, then `lifting' them to metric spaces.
For this, we will need to consider dynamic $p$-morphisms from dynamic metric systems to dynamic derivative frames. 
The  statement below follows directly by unravelling the definitions.

\begin{proposition}\label{propDMorphMetric}
Let $\mathfrak X=\langle X,\delta,f\rangle$ be a dynamic metric system and let $\mathfrak M=\langle  W, \sqsubset, g \rangle$ be a dynamic derivative frame.
Let $\pi:X\to M$. Then $\pi$ is a dynamic $p$-morphism if for each $x\in X$ and $w\in W$,
\begin{itemize}
    \item  (forth condition) if $w=\pi(x)$, then there exists $\varepsilon>0$ such that, for all $y\in X$, $0<\delta(x,y)<\varepsilon$ implies  $\pi(x)\sqsubset \pi(y)$. 
    \item (back condition) if $w=\pi(x)$ and $w\sqsubset v$, for some $v\in W$, then there exists $y\in X$ with $0<d(x,y)<\varepsilon$ such that $\pi(y)=v$.
    \item $\pi\circ f= g\circ \pi$.
    

\end{itemize}
\end{proposition}
Such morphisms between dynamic derivative metric spaces to dynamic derivative frames will be used explicitly in Section \ref{topologicaldcompleteness}.
For now, we provide the following example to illustrate them.
\begin{example}\label{metrictokripke}
The map $\pi$ in Figure \ref{fig:metrictokripke} illustrates a $p$-morphism, where the centre of the circle on the left is mapped to the root of the Kripke frame on the right, the rays in red are mapped to the intermediate red points in the Kripke frame, and the open regions in grey are mapped to the leafs of the Kripke frame.
The rotation dynamics of the circle are translated to cycles in the Kripke frame marked in dashed lines.
\begin{figure}[H]
    \centering
   \includegraphics[scale=0.27]{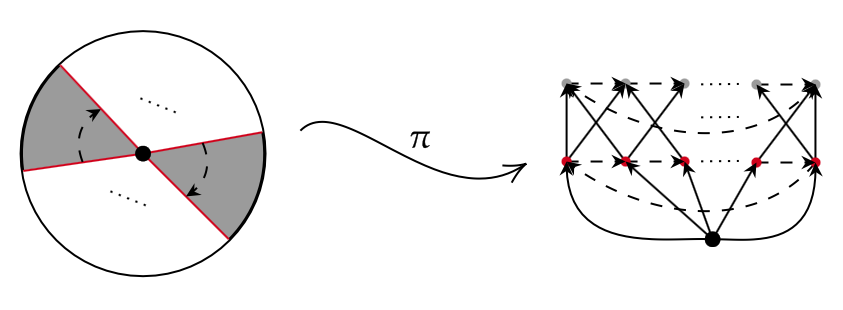}
    \caption{A $p$-morphism from a dynamic metric space to a dynamic Kripke frame.}
    \label{fig:metrictokripke}
\end{figure}
\end{example}
\section{The canonical model}

The first step in our Kripke completeness proof will be a fairly standard canonical model construction.
A \emph{maximal $\Lambda$-consistent set} ($\Lambda$-MCS) $w$ is a set of formulas that is $\Lambda$-consistent, i.e.\ $w\not\vdash_{\Lambda} \bot$, and every set of formulas that properly contains it is $\Lambda$-inconsistent.

Given a logic $\Lambda$ over $\mathcal L^{\circ}_\dd$, let $\mathfrak{M}^\Lambda_{\rm c}=\langle W_{\rm c}, \sqsubset _{\rm c} ,g_{\rm c},\nu_{\rm c}\rangle$ be the canonical model of $\Lambda$, where
 \begin{enumerate}
	\item $W_{\rm c}$ is the set of all $\Lambda$-MCSs;
	\item $w\sqsubset _{\rm c}v$ iff for all formulas $\varphi$, if $\square\varphi\in w$, then $\varphi\in v$;

\item $g_{\rm c}(w)=\{\varphi:\bc\varphi\in w\}$;
\item $\nu_{\rm c}(p)=\{w:p\in w\}$.
\end{enumerate}

It can easily be verified that $\mathbf{K4C}$ defines the class of transitive, monotonic Kripke models. 
Moreover, $\mathbf{K4I}$ defines the class of all transitive, strictly monotonic Kripke models. 
We call these models $\mathbf{K4C}$ models and $\mathbf{K4I}$ models, respectively.

\begin{lemma}\
\begin{enumerate}

\item If $\Lambda$ extends $\mathbf{K4C}$, then the canonical model of $\Lambda$ is transitive and monotonic.

\item If $\Lambda$ extends $\mathbf{K4I}$, then the canonical model of $\Lambda$ is transitive and strictly monotonic.

\item If $\Lambda$ extends $\mathbf{K4D}$, then the canonical model of $\Lambda$ is serial with respect to $\sqsubset$ (i.e., for all $w \in W_{\rm c}$, there is $v\sqsupset w$).

\end{enumerate}
\end{lemma}
\begin{proof}
Let $\mathfrak{M}^\Lambda_{\rm c}=\langle W,\sqsubset, g, \nu \rangle$. 
The proof of the first statement appears in \cite{CSLpaper}, and the third is standard (see e.g.~\cite{black}).
For the second statement, showing that $\sqsubset$ is transitive is routine and follows from the axiom ${\rm 4}$. 
Suppose that $\Lambda$ extends $\mathbf{K4I}$. We prove that $g$ is strictly monotonic. 
Suppose that $w\sqsubset  v$. 
We consider an arbitrary $\square\psi\in g (w)$. 
By definition $\bc\square \psi\in w$ and so by ${\rm C_{\dd}}$ we have $\square\bc \psi\in v$.
Since  $w\sqsubset  v$, then $\bc\psi\in (v) $ and hence $\psi\in g(v)$. 
It follows that $g(w) \sqsubset   g(v)$, as required, thus $g$ is monotonic.
\end{proof}

The proofs of the following two lemmas are standard and can be found for example in \cite{black}.

\begin{lemma}[existence lemma]
Let $\Lambda$ be a normal modal logic and let $\mathfrak M^\Lambda_{\rm c} = \langle W_{\rm c},\sqsubset  _{\rm c},g_{\rm c},\nu_{\rm c}\rangle$.
Then, for every $w\in W_{\rm c}$ and every formula $\varphi$ in $\Lambda$, if $\dd\varphi\in w$ then there exists a point $v\in W_{\rm c}$ such that $w\sqsubset  _{\rm c} v$ and $\varphi\in v$.
\end{lemma}

This is already enough to obtain a standard truth lemma for tangle-free languages, and thus completeness.

 \begin{lemma}[truth lemma]\label{lemmTruth}
 Let $\Lambda$ be a normal modal logic without tangle.
For every $w\in W_{\rm c}$ and every formula $\varphi$ in $\Lambda$, $$\mathfrak M^\Lambda_{\rm c}, w\models \varphi \text{ iff }\varphi\in w.$$
 \end{lemma}

\begin{corollary}
The logic $\mathbf{K4C}$ is sound and complete with respect to the class of all transitive and monotonic dynamic derivative frames, and $\mathbf{K4I}$ is sound and complete with respect to the class of all transitive and strictly monotonic dynamic derivative frames.
The logics $\mathbf{K4DC}$ and $\mathbf{K4DI}$ are sound and complete for the respective classes of serial frames.
\end{corollary}

Thus we obtain completeness for tangle-free logics in a standard way.
However, Lemma \ref{lemmTruth} fails in the presence of tangle, and so we will have to work a bit harder to achieve completeness in this setting.

\color{black}

\section{A finitary accessibility relation}

One key ingredient in our finite model property proof will be the construction of a `finitary' accessibility relation $\sqsubset  _{\Phi}$ on the canonical model.
This accessibility relation will have the property that each point has finitely many successors, yet the existence lemma will hold for formulas in a prescribed finite set $\Phi$.

\begin{definition}[$\varphi$-final set]\label{finalset}
Fix a logic $\Lambda$. A set $w$ is said to be a \emph{$\varphi$-final set} (or point) if $w$ is a $\Lambda$-MCS, $\varphi\in w$, and whenever $w \sqsubset  _{\rm c} v$ and $\varphi\in v$, it follows that $v\in C(w)$.
\end{definition}

Let $\sqsubseteq_{\rm c}$ be the reflexive closure of $\sqsubset_{\rm c}$.
It will be convenient to characterise $\sqsubseteq_{\rm c}$ in the canonical model syntactically.
Recall that $\boxdot\varphi:= \varphi\wedge\square \varphi$.
The following is proven in~\cite{CSLpaper} but goes back to Fine~\cite{Fin74c}.

\begin{lemma}\label{choice2}
If $\dd\varphi\in w$, then there is $\varphi$-final point $v$ such that $w \sqsubset  _{\rm c} v$.
\end{lemma}

We are now ready to prove the main result of this section regarding the existence of the finitary relation $\sqsubset_\Phi$.

\begin{lemma}\label{lemmPhiConds}
Let $\Lambda$ extend $\mathbf{K4^{\infty}}$ and let $\Phi$ be a finite set of formulas closed under subformulas. 
There is an auxiliary relation $\sqsubset  _{\Phi} $ on the canonical model of $\Lambda$ such that:
\begin{enumerate}[(i)]
\item\label{condOne} $\sqsubset  _{\Phi}$ is a subset of $\sqsubset  _{\rm c}$.
\item\label{condTwo} For each $w\in W$, the set ${\sqsubset  _{\Phi}}(w)$ of $\sqsubset  _{\Phi}$-successors of $w$ is finite.
\item\label{condThree} If $\dd\varphi\in w\cap\Phi$, then there exists $v\in W$ with $w\sqsubset  _{\Phi}v$ and $\varphi\in v$.


\item\label{condNew}  If $\dtan \Psi\in w\cap\Phi$, then there exists $\{v_\psi :\psi\in \Psi\}\subseteq W$ (not necessarily distinct) with $w\sqsubseteq  _{\Phi}v_\psi$, $\psi, \dtan \Psi \in v_\psi $, and $v_\psi\sqsubset  v_\varphi$ for all $\psi,\varphi\in \Psi$ (including $\psi=\varphi$).

\item\label{condFour} If $w \equiv  _{\rm c} v  $ then ${\sqsubset  _{\Phi}(w)} = {\sqsubseteq  _{\Phi}(v)}  $.
\item\label{condFive} $\sqsubset  _{\Phi}$ is transitive.
\end{enumerate}
\end{lemma}

\begin{proof}
Let $C$ be any cluster of points in $W$ and define 
$$ { \sqsubset  _{\rm c}}(C)=\bigcup\{{\sqsubset  _{\rm c}}(v):v\in C\}.$$
 We construct the transitive relation $\sqsubset  _{\Phi}$ as follows:
 Using Lemma \ref{choice2}, we use the axiom of choice to choose a function that for each formula $\varphi$ and each cluster $C$ such that $\dd\varphi \in \bigcup C$, assigns a $\varphi$-final point $w(\varphi,C)$ such that $w(\varphi,C)\in { \sqsubset  _{\rm c}}(C)$.
Similarly, if $\dtan\Gamma \in \bigcup C$, we choose a $\dtan\Gamma $-final point $v \in { \sqsubset  _{\rm c}}(C)$.
For each $\varphi\in\Gamma$, we have that $\vdash\dtan\Gamma\to \dd(\varphi\wedge\dtan\Gamma)$, and we may choose $(\varphi\wedge\dtan\Gamma)$-final $v(\varphi,\dtan\Gamma,C) \sqsupseteq_{\rm c} v$.
Since $v$ was already $\dtan\Gamma$-final, we must have that $v(\varphi,\dtan\Gamma,C)$ is in the same cluster as $v$.

\sloppy Set $u \sqsubset  _{\Phi}^0 u'$ iff $u\sqsubset_{\rm c} u'$ and there exists $\dd \psi\in u\cap \Phi$ such that $u' = w(\psi,C(u)) $, or there are $\dtan\Gamma\in u\cap \Phi$ and $\varphi\in \Gamma$ such that $u'= v(\varphi,\dtan\Gamma,C(u))$.
 Let $\sqsubset  _{\Phi}$ be the transitive closure of $\sqsubset  _{\Phi}^0$.

 It is clear that \ref{condOne}, \ref{condThree}, \ref{condNew}, \ref{condFour}, and \ref{condFive} follow directly from the construction: \ref{condOne} follows from the fact that $\sqsubset_\Phi$ is the transitive closure of $\sqsubset^0_\Phi$; \ref{condThree} and \ref{condNew} follow from Lemma \ref{finalset}; \ref{condFour} follows from the fact that $v\in C(w)$ and by assuming that $v\sqsubseteq_\Phi u$ and unravelling the definition of $\sqsubseteq_\Phi$; \ref{condFive} follows from the definition of transitive closure.
 
We continue to verify condition \ref{condTwo}. First observe that ${\sqsubset_{\Phi}^0}(u)$ is finite by construction, as it contains at most one element for each $\varphi\in \Phi$ and another for each $\varphi\in \Gamma$ with $\dtan\Gamma\in \Phi$.
Now, if $u \sqsubset  _{\Phi}v$, then there is a sequence
$$ u\sqsubset  _{\Phi}^0v_1\sqsubset  _{\Phi}^0\dots \sqsubset  _{\Phi}^0v_n=v.$$
By taking a minimal such sequence, we may assume that it is injective.
Consider the tree consisting of all such sequences (ordered by the initial segment relation).
This is a finitely-branching tree, as ${\sqsubset_{\Phi}^0}(x)$ is always finite.
Moreover, if ${\sqsubset_{\Phi}}(u)$ is infinite, then this tree is infinite.
By K\"onig's lemma, there is an infinite sequence
$$ u \sqsubset  _{\Phi}^0v_1\sqsubset  _{\Phi}^0 v_2\sqsubset  _{\Phi}^0 \dots .$$

By definition of $\sqsubset_{\Phi}^0$, for each $i \in  \omega $ there is $\varphi_i \in \Phi$ such that $v_{i+1}$ is $\varphi_i$-final.
Since ${\sqsubset_{\Phi}^0} \subseteq {\sqsubset_{\rm c}}$, we have that $v_{i} \sqsubseteq_{\rm c} v_j$ whenever $i\leq j$.
Since $\Phi$ is finite, there is some $\theta\in\Phi$ such that $v_{i}$ is $\theta$-final for infinitely many values of $i$.
Let $i_0$ be the least such value.
If $i>i_0$ is any other such value, $v_{i_0} \sqsubset_{\rm c} v_i$ together with $\theta$-finality of $v_{i_0}$ yields $v_i \sqsubset_{\rm c} v_{i_0}$.
Thus $v_i\in C(v_{i_0})$ and $ v_{i} \sqsubset_{\Phi}^0 v_{i+1}$, which by definition of $\sqsubset^0_{\Phi}$ yields $  v_{i+1} \in {\sqsubset_{\Phi}^0} (v_{i_0} )$.
But ${\sqsubset_{\Phi}^0} (v_{i_0} )$ is finite, contradicting that the chain is infinite and injective.
\end{proof}
\section{Stories and $\Phi$-morphisms}
In this subsection we show that the logics $\mathbf{K4C^{\infty}}$, $\mathbf{K4DC^{\infty}}$, $\mathbf{K4I^{\infty}}$, and $\mathbf{K4DI^{\infty}}$ have the finite model property by constructing finite models and truth preserving maps from these models to the canonical model. 

If $\sqsubset$ is a transitive relation on $A$, $\lb A,\sqsubset \rb$ is called \emph{tree-like} if whenever $a\sqsubseteq c$ and $b\sqsubseteq c$, it follows that $a\sqsubseteq b$ or $b\sqsubseteq a$.
We will use labelled tree-like structures called {\em moments} to record the `static' information at a point; that is, the structure involving $\sqsubset$, but not $f$.

\begin{definition}[moment]
Fix $\Lambda\in \{{\bf K4},{\bf K4D}\}$.
A {\em $\Lambda$-moment} is a structure $\mathfrak m = \langle |\mathfrak m|,\sqsubset  _\mathfrak m,\nu_\mathfrak m,r_\moment \rangle $, where ${\langle |\mathfrak m|,\sqsubset  _\mathfrak m \rangle}$ is a finite tree-like $\Lambda$-frame with a root $r_\moment$, and $\nu_\mathfrak m$ is a valuation on $|\mathfrak m|$.\end{definition}

In order to also record `dynamic' information, i.e.~information involving the transition function, we will stack up several moments together to form a `story'.
Below, $ \bigsqcup$ denotes a disjoint union and $f[S]$ denotes the image of a set $S$ under the map $f$.

\begin{definition}[story and immersive story]\label{defStory}
A {\em story (with duration $I$)} is a structure  $\gog = \langle|\gog|,\sqsubset  _\gog,f_\gog,\nu_\gog,r_\gog \rangle$ such that there exist $I< \omega$, moments $\gog_i = \langle |\gog_i|,\sqsubset  _i,\nu_i,r_i \rangle$ for each $i\leq I$, and functions $(f_i)_{i<I}$ such that:
\begin{enumerate}

\item $|\gog| = \bigsqcup_{i\leq I} |\gog_i| $;

\item $\sqsubset  _\gog = \bigsqcup_{i\leq I} \sqsubset  _i $;

\item $\nu_\gog(p) = \bigsqcup_{i\leq I} \nu_i(p) $ for each variable $p$;

\item $r_\gog =r_0$;

\item $f_\gog = {\rm Id}_{I} \cup \bigsqcup_{i <I} f_i $ with $f_i\colon |\gog_i| \to |\gog_{i+1}|$ being a monotonic map such that $f_\gog$ is
\begin{description}

\item[root preserving:] $f_i(r_i) = r_{i+1}$ for all $i<I$,

\item[almost injective:] for every $x,y\in|\gog_i|$, if $f_\gog(x) = f_\gog (y)$ then $f_\gog(x)$ is irreflexive,

\item[cluster-preserving:] for every $x\in |\gog|$, $C(f_\gog(x)) = f_\gog[C(x)] $, and

\item[stabilising:] ${ f}_{I}$ is the identity on $|\gog_I|$.

\end{description}
\end{enumerate}
If moreover each $f_i$ is strictly monotonic and injective, we say that $\gog$ is an {\em immersive story.}
If each $\gog_i $ is a $\Lambda$-moment we say that $\gog$ is a {\em $\Lambda{\bf C}$-story,} and if $\gog$ is immersive we say that $\gog$ is a {\em $\Lambda{\bf I}$-story.}
\end{definition}

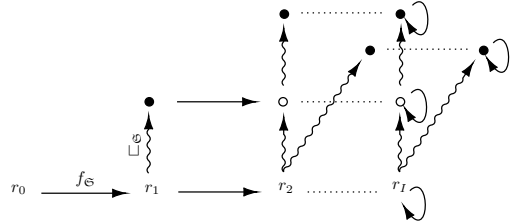
\begin{figure}[H]

\centering

\scalebox{0.63}{%
\tikzset{every picture/.style={line width=0.75pt}} 

\begin{tikzpicture}[x=0.75pt,y=0.75pt,yscale=-1,xscale=1]

\draw    (458.5,199) .. controls (475.25,182.26) and (475.5,246.04) .. (454.48,212.6) ;
\draw [shift={(453.5,211)}, rotate = 419.26] [color={rgb, 255:red, 0; green, 0; blue, 0 }  ][fill={rgb, 255:red, 0; green, 0; blue, 0 }  ][line width=0.75]    (10.93,-3.29) .. controls (6.95,-1.4) and (3.31,-0.3) .. (0,0) .. controls (3.31,0.3) and (6.95,1.4) .. (10.93,3.29)   ;
\draw    (166.5,201) -- (231.5,201) ;
\draw [shift={(233.5,201)}, rotate = 180] [color={rgb, 255:red, 0; green, 0; blue, 0 }  ][fill={rgb, 255:red, 0; green, 0; blue, 0 }  ][line width=0.75]    (10.93,-3.29) .. controls (6.95,-1.4) and (3.31,-0.3) .. (0,0) .. controls (3.31,0.3) and (6.95,1.4) .. (10.93,3.29)   ;
\draw    (251.5,185) .. controls (249.83,183.33) and (249.83,181.67) .. (251.5,180) .. controls (253.17,178.33) and (253.17,176.67) .. (251.5,175) .. controls (249.83,173.33) and (249.83,171.67) .. (251.5,170) .. controls (253.17,168.33) and (253.17,166.67) .. (251.5,165) .. controls (249.83,163.33) and (249.83,161.67) .. (251.5,160) .. controls (253.17,158.33) and (253.17,156.67) .. (251.5,155) .. controls (249.83,153.33) and (249.83,151.67) .. (251.5,150) -- (251.5,149) -- (251.5,141) ;
\draw [shift={(251.5,139)}, rotate = 450] [color={rgb, 255:red, 0; green, 0; blue, 0 }  ][fill={rgb, 255:red, 0; green, 0; blue, 0 }  ][line width=0.75]    (10.93,-3.29) .. controls (6.95,-1.4) and (3.31,-0.3) .. (0,0) .. controls (3.31,0.3) and (6.95,1.4) .. (10.93,3.29)   ;
\draw    (358.5,184) .. controls (356.83,182.33) and (356.83,180.67) .. (358.5,179) .. controls (360.17,177.33) and (360.17,175.67) .. (358.5,174) .. controls (356.83,172.33) and (356.83,170.67) .. (358.5,169) .. controls (360.17,167.33) and (360.17,165.67) .. (358.5,164) .. controls (356.83,162.33) and (356.83,160.67) .. (358.5,159) .. controls (360.17,157.33) and (360.17,155.67) .. (358.5,154) -- (358.5,150) -- (358.5,142) ;
\draw [shift={(358.5,140)}, rotate = 450] [color={rgb, 255:red, 0; green, 0; blue, 0 }  ][fill={rgb, 255:red, 0; green, 0; blue, 0 }  ][line width=0.75]    (10.93,-3.29) .. controls (6.95,-1.4) and (3.31,-0.3) .. (0,0) .. controls (3.31,0.3) and (6.95,1.4) .. (10.93,3.29)   ;
\draw    (450.5,184) .. controls (448.83,182.33) and (448.83,180.67) .. (450.5,179) .. controls (452.17,177.33) and (452.17,175.67) .. (450.5,174) .. controls (448.83,172.33) and (448.83,170.67) .. (450.5,169) .. controls (452.17,167.33) and (452.17,165.67) .. (450.5,164) .. controls (448.83,162.33) and (448.83,160.67) .. (450.5,159) .. controls (452.17,157.33) and (452.17,155.67) .. (450.5,154) -- (450.5,150) -- (450.5,142) ;
\draw [shift={(450.5,140)}, rotate = 450] [color={rgb, 255:red, 0; green, 0; blue, 0 }  ][fill={rgb, 255:red, 0; green, 0; blue, 0 }  ][line width=0.75]    (10.93,-3.29) .. controls (6.95,-1.4) and (3.31,-0.3) .. (0,0) .. controls (3.31,0.3) and (6.95,1.4) .. (10.93,3.29)   ;
\draw    (252,128) ;
\draw [shift={(252,128)}, rotate = 0] [color={rgb, 255:red, 0; green, 0; blue, 0 }  ][fill={rgb, 255:red, 0; green, 0; blue, 0 }  ][line width=0.75]      (0, 0) circle [x radius= 3.35, y radius= 3.35]   ;
\draw    (358,128) ;
\draw [shift={(358,128)}, rotate = 0] [color={rgb, 255:red, 0; green, 0; blue, 0 }  ][line width=0.75]      (0, 0) circle [x radius= 3.35, y radius= 3.35]   ;
\draw    (358.5,115) .. controls (356.83,113.33) and (356.83,111.67) .. (358.5,110) .. controls (360.17,108.33) and (360.17,106.67) .. (358.5,105) .. controls (356.83,103.33) and (356.83,101.67) .. (358.5,100) .. controls (360.17,98.33) and (360.17,96.67) .. (358.5,95) .. controls (356.83,93.33) and (356.83,91.67) .. (358.5,90) .. controls (360.17,88.33) and (360.17,86.67) .. (358.5,85) .. controls (356.83,83.33) and (356.83,81.67) .. (358.5,80) -- (358.5,79) -- (358.5,71) ;
\draw [shift={(358.5,69)}, rotate = 450] [color={rgb, 255:red, 0; green, 0; blue, 0 }  ][fill={rgb, 255:red, 0; green, 0; blue, 0 }  ][line width=0.75]    (10.93,-3.29) .. controls (6.95,-1.4) and (3.31,-0.3) .. (0,0) .. controls (3.31,0.3) and (6.95,1.4) .. (10.93,3.29)   ;
\draw    (359,58) ;
\draw [shift={(359,58)}, rotate = 0] [color={rgb, 255:red, 0; green, 0; blue, 0 }  ][fill={rgb, 255:red, 0; green, 0; blue, 0 }  ][line width=0.75]      (0, 0) circle [x radius= 3.35, y radius= 3.35]   ;
\draw    (451,128) ;
\draw [shift={(451,128)}, rotate = 0] [color={rgb, 255:red, 0; green, 0; blue, 0 }  ][line width=0.75]      (0, 0) circle [x radius= 3.35, y radius= 3.35]   ;
\draw    (451.5,115) .. controls (449.83,113.33) and (449.83,111.67) .. (451.5,110) .. controls (453.17,108.33) and (453.17,106.67) .. (451.5,105) .. controls (449.83,103.33) and (449.83,101.67) .. (451.5,100) .. controls (453.17,98.33) and (453.17,96.67) .. (451.5,95) .. controls (449.83,93.33) and (449.83,91.67) .. (451.5,90) .. controls (453.17,88.33) and (453.17,86.67) .. (451.5,85) .. controls (449.83,83.33) and (449.83,81.67) .. (451.5,80) -- (451.5,79) -- (451.5,71) ;
\draw [shift={(451.5,69)}, rotate = 450] [color={rgb, 255:red, 0; green, 0; blue, 0 }  ][fill={rgb, 255:red, 0; green, 0; blue, 0 }  ][line width=0.75]    (10.93,-3.29) .. controls (6.95,-1.4) and (3.31,-0.3) .. (0,0) .. controls (3.31,0.3) and (6.95,1.4) .. (10.93,3.29)   ;
\draw    (451,58) ;
\draw [shift={(451,58)}, rotate = 0] [color={rgb, 255:red, 0; green, 0; blue, 0 }  ][fill={rgb, 255:red, 0; green, 0; blue, 0 }  ][line width=0.75]      (0, 0) circle [x radius= 3.35, y radius= 3.35]   ;
\draw    (358.5,184) .. controls (358.07,181.68) and (359.02,180.31) .. (361.34,179.88) .. controls (363.66,179.46) and (364.61,178.09) .. (364.18,175.77) .. controls (363.75,173.45) and (364.7,172.08) .. (367.02,171.65) .. controls (369.33,171.22) and (370.28,169.85) .. (369.85,167.54) .. controls (369.42,165.22) and (370.37,163.85) .. (372.69,163.42) .. controls (375.01,162.99) and (375.96,161.62) .. (375.53,159.3) .. controls (375.1,156.98) and (376.05,155.61) .. (378.37,155.19) .. controls (380.69,154.76) and (381.64,153.39) .. (381.21,151.07) .. controls (380.78,148.75) and (381.73,147.38) .. (384.05,146.96) .. controls (386.37,146.53) and (387.32,145.16) .. (386.89,142.84) .. controls (386.46,140.52) and (387.41,139.15) .. (389.73,138.72) .. controls (392.04,138.29) and (392.99,136.92) .. (392.56,134.61) .. controls (392.13,132.29) and (393.08,130.92) .. (395.4,130.49) .. controls (397.72,130.07) and (398.67,128.7) .. (398.24,126.38) .. controls (397.81,124.06) and (398.76,122.69) .. (401.08,122.26) .. controls (403.4,121.83) and (404.35,120.46) .. (403.92,118.14) .. controls (403.49,115.82) and (404.44,114.45) .. (406.76,114.03) .. controls (409.08,113.6) and (410.03,112.23) .. (409.6,109.91) .. controls (409.17,107.59) and (410.11,106.22) .. (412.43,105.79) -- (412.82,105.23) -- (417.36,98.65) ;
\draw [shift={(418.5,97)}, rotate = 484.59] [color={rgb, 255:red, 0; green, 0; blue, 0 }  ][fill={rgb, 255:red, 0; green, 0; blue, 0 }  ][line width=0.75]    (10.93,-3.29) .. controls (6.95,-1.4) and (3.31,-0.3) .. (0,0) .. controls (3.31,0.3) and (6.95,1.4) .. (10.93,3.29)   ;
\draw    (450.5,184) .. controls (450.03,181.69) and (450.96,180.3) .. (453.27,179.84) .. controls (455.58,179.38) and (456.51,177.99) .. (456.05,175.68) .. controls (455.58,173.37) and (456.51,171.98) .. (458.82,171.52) .. controls (461.13,171.06) and (462.06,169.67) .. (461.59,167.36) .. controls (461.13,165.05) and (462.06,163.66) .. (464.37,163.2) .. controls (466.68,162.74) and (467.61,161.35) .. (467.14,159.04) .. controls (466.67,156.73) and (467.6,155.34) .. (469.91,154.88) .. controls (472.22,154.42) and (473.15,153.03) .. (472.69,150.72) .. controls (472.22,148.41) and (473.15,147.02) .. (475.46,146.56) .. controls (477.77,146.1) and (478.7,144.71) .. (478.24,142.4) .. controls (477.77,140.09) and (478.7,138.7) .. (481.01,138.24) .. controls (483.32,137.78) and (484.25,136.39) .. (483.78,134.08) .. controls (483.32,131.77) and (484.25,130.38) .. (486.56,129.92) .. controls (488.87,129.46) and (489.8,128.07) .. (489.33,125.76) .. controls (488.86,123.45) and (489.79,122.06) .. (492.1,121.6) .. controls (494.41,121.14) and (495.34,119.75) .. (494.88,117.44) .. controls (494.41,115.13) and (495.34,113.74) .. (497.65,113.28) .. controls (499.96,112.82) and (500.89,111.43) .. (500.42,109.12) -- (502.95,105.32) -- (507.39,98.66) ;
\draw [shift={(508.5,97)}, rotate = 483.69] [color={rgb, 255:red, 0; green, 0; blue, 0 }  ][fill={rgb, 255:red, 0; green, 0; blue, 0 }  ][line width=0.75]    (10.93,-3.29) .. controls (6.95,-1.4) and (3.31,-0.3) .. (0,0) .. controls (3.31,0.3) and (6.95,1.4) .. (10.93,3.29)   ;
\draw  [dash pattern={on 0.84pt off 2.51pt}]  (377.5,200) -- (432.5,200) ;
\draw  [dash pattern={on 0.84pt off 2.51pt}]  (372.5,57) -- (437.5,57) ;
\draw  [dash pattern={on 0.84pt off 2.51pt}]  (372.5,128) -- (437.5,128) ;
\draw    (427,87) ;
\draw [shift={(427,87)}, rotate = 0] [color={rgb, 255:red, 0; green, 0; blue, 0 }  ][fill={rgb, 255:red, 0; green, 0; blue, 0 }  ][line width=0.75]      (0, 0) circle [x radius= 3.35, y radius= 3.35]   ;
\draw    (517,87) ;
\draw [shift={(517,87)}, rotate = 0] [color={rgb, 255:red, 0; green, 0; blue, 0 }  ][fill={rgb, 255:red, 0; green, 0; blue, 0 }  ][line width=0.75]      (0, 0) circle [x radius= 3.35, y radius= 3.35]   ;
\draw  [dash pattern={on 0.84pt off 2.51pt}]  (440.5,86) -- (505.5,86) ;
\draw    (274,128) -- (334.5,128) ;
\draw [shift={(336.5,128)}, rotate = 180] [color={rgb, 255:red, 0; green, 0; blue, 0 }  ][fill={rgb, 255:red, 0; green, 0; blue, 0 }  ][line width=0.75]    (10.93,-3.29) .. controls (6.95,-1.4) and (3.31,-0.3) .. (0,0) .. controls (3.31,0.3) and (6.95,1.4) .. (10.93,3.29)   ;
\draw    (524.5,81) .. controls (541.25,64.26) and (541.5,128.04) .. (520.48,94.6) ;
\draw [shift={(519.5,93)}, rotate = 419.26] [color={rgb, 255:red, 0; green, 0; blue, 0 }  ][fill={rgb, 255:red, 0; green, 0; blue, 0 }  ][line width=0.75]    (10.93,-3.29) .. controls (6.95,-1.4) and (3.31,-0.3) .. (0,0) .. controls (3.31,0.3) and (6.95,1.4) .. (10.93,3.29)   ;
\draw    (460.5,51) .. controls (477.25,34.25) and (477.5,98.04) .. (456.48,64.6) ;
\draw [shift={(455.5,63)}, rotate = 419.26] [color={rgb, 255:red, 0; green, 0; blue, 0 }  ][fill={rgb, 255:red, 0; green, 0; blue, 0 }  ][fill={rgb, 255:red, 0; green, 0; blue, 0 }  ][line width=0.75]    (10.93,-3.29) .. controls (6.95,-1.4) and (3.31,-0.3) .. (0,0) .. controls (3.31,0.3) and (6.95,1.4) .. (10.93,3.29)   ;
\draw    (459.5,121) .. controls (476.25,104.26) and (476.5,168.04) .. (455.48,134.6) ;
\draw [shift={(454.5,133)}, rotate = 419.26] [color={rgb, 255:red, 0; green, 0; blue, 0 }  ][fill={rgb, 255:red, 0; green, 0; blue, 0 }  ][line width=0.75]    (10.93,-3.29) .. controls (6.95,-1.4) and (3.31,-0.3) .. (0,0) .. controls (3.31,0.3) and (6.95,1.4) .. (10.93,3.29)   ;
\draw    (275,200) -- (335.5,200) ;
\draw [shift={(337.5,200)}, rotate = 180] [color={rgb, 255:red, 0; green, 0; blue, 0 }  ][fill={rgb, 255:red, 0; green, 0; blue, 0 }  ][line width=0.75]    (10.93,-3.29) .. controls (6.95,-1.4) and (3.31,-0.3) .. (0,0) .. controls (3.31,0.3) and (6.95,1.4) .. (10.93,3.29)   ;

\draw (247,192.4) node [anchor=north west][inner sep=0.75pt]    {$r_{1}$};
\draw (353,191.4) node [anchor=north west][inner sep=0.75pt]    {$r_{2}$};
\draw (443,190.4) node [anchor=north west][inner sep=0.75pt]    {$r_{I}$};
\draw (141,192.4) node [anchor=north west][inner sep=0.75pt]    {$r_{0}$};
\draw (192,182.4) node [anchor=north west][inner sep=0.75pt]    {$f_{\mathfrak S }$};
\draw (234.5,173.5) node [anchor=north west][inner sep=0.75pt]  [rotate=-270]  {$\sqsubset _{\mathfrak S }$};
\end{tikzpicture}
}
\vspace{-2mm}
\caption[story]{An example of a $\mathbf{K4D}$-story. The squiggly arrows represent the relation $\sqsubset_{\mathfrak  S}$ while the straight  arrows represent the function $f_{\mathfrak  S}$. Black points are reflexive with respect to $\sqsubset_{\mathfrak  S}$.
Each vertical slice represents a $\mathbf{K4D}$-moment.}
\label{fig:story}
\end{figure}
We often omit the index $\mathfrak m$ or $\gog$ when this does not lead to confusion.
We may also assign different notations to the components of a moment, so that if we write $\mathfrak m= \langle W,\sqsubset,\nu,x\rangle $, it is understood that $W=|\mathfrak m|$, ${\sqsubset  }= {\sqsubset  _\mathfrak m}$, etc.

\ignore{
\begin{definition}(\emph{story}) We define a story with duration $I$ as a model
$$ {\gog}=\langle |{\gog}|,{\prec},{f},{\nu}\rangle,$$
where $|{\gog}|\neq\varnothing$ is a finite set of points, ${\prec}$ is a weakly transitive and monotonic relation, $C^\prec(w) $ denotes the $\prec$ cluster of $w$, the reflexive closure of $\prec$ is denoted by $\preceq$, and
\begin{enumerate}
\item $|{\gog}|=\bigsqcup_{i\leq I}|{\gog}|_i$, where each $|{\gog}|_i$ is open and has a cluster-root, i.e.\ there is $C^-(w)\subseteq |{\gog}|_i$ such that $w
\prec v$ for all $w\in C^-(w)$ and for all $v\in |{\gog}|_i$ where $v\neq w$. The cluster-root of $|{\gog}|_0$ is called the cluster-root of the story, or simply the root of the story;
\item ${f}:|{\gog}|\rightarrow|{\gog}|$ is a monotonic function such that for all $i<I$, $f|{\gog}|_i\subseteq |{\gog}|_{i+1}$ and $f\upharpoonright |{\gog}|_I$ is the identity map;
\item ${\nu}$ is an evaluation function assigning to each propositional variable a subset of $|{\gog}|$;
\item $x \prec y$ implies that $x,y\in|{\gog}|_i$ for some $i\leq I$;
\item  If $x\prec f(y)$ then there exists $z \preceq y$ such that $f(z)\in C^\prec(x)$.

 
\end{enumerate}
\end{definition}
}

The finitary accessibility relation $\sqsubset  _{\Phi}$ will allow us to weaken the conditions on $p$-morphisms and still obtain maps that preserve the truth of (some) formulas.

\begin{definition}[$\Phi$-morphism]\label{phimorstatic}
Fix a logic $\Lambda$. Let $\mathfrak M^\Lambda_{\rm c} = \langle W_{\rm c},\sqsubset_{\rm c},g_{\rm c},\nu_{\rm c}\rangle$ and let $\mathfrak m$ be a moment.
A map ${\pi}:|\moment |\rightarrow W_{\rm c}$ is called a \emph{$\Phi$-morphism} if for all $x\in |\moment|$ the following conditions are satisfied:
\begin{enumerate}
\item \label{itAtoms} $x\in {\nu_\moment} (p)  \iff p \in {\pi}(x)$;

\item \label{itForth} If $ x \sqsubset  _\moment y$ then $\pi(x) \sqsubset  _{\rm c}  \pi(y)$;

\item\label{itBack} If ${\pi}(x)\sqsubset  _{\Phi}v$ for some $v\in W_{\rm c}$, then there exists $y\in|\moment|$ such that
$x \sqsubset  _\moment y  \text{ and } v={\pi}(y)$.

\end{enumerate}
We say that $\pi$ is {\em distinguished} if whenever $x$ is reflexive, then either $\pi(x)$ is $\theta$-final for some $\theta \in \mathcal{L}_{\dd\dtan}^{\circ }$ (not necessarily in $\Phi$), or else there is $y\equiv_\moment x$ such that $\pi(y)\neq \pi(x)$.
\end{definition}

We also need a dynamic variant of a $\Phi$-morphism which takes the transition function into account.

\begin{definition}[dynamic $\Phi$-morphism]\label{phimor}
Fix a logic $\Lambda$. Let $\mathfrak M^\Lambda_{\rm c} = \langle W_{\rm c},\sqsubset_{\rm c},g_{\rm c},\nu_{\rm c}\rangle$ and let $\mathfrak S$ be a story of duration $I$.
A map ${\pi}:|{\gog}|\rightarrow W_{\rm c}$ is called a \emph{dynamic $\Phi$-morphism} if for all $i\leq I$, $\pi\upharpoonright |\gog_i|$ is a $\Phi$-morphism and if $x\in|\gog_i|$ for some $i<I$, then $g_{\rm c}( {\pi}(x))={\pi}( f_{\gog}(x))$. 

We say that $\pi$ is a \emph{distinguished dynamic $\Phi$-morphism} if each $\pi\upharpoonright |\gog_i|$ is distinguished and for every reflexive $y\in |\gog_i|$, either $y$ is $\theta$-final for some $\theta$, or $i>0$ and $y\in f_\gog[|\gog_{i-1}|]$.
\end{definition}

We now show that a dynamic $\Phi$-morphism $\pi$ preserves the truth of formulas of suitable $\bc$-depth.
In order to prove this, we need the following witnessing lemma.

\begin{lemma}[witnessing lemma]
\label{lemWit}
Let $\mathfrak S$ be a weak story of duration $I$ and let $x \in |\mathfrak S_i|$ be with $i\leq I$. 
Let $\Phi$ be closed under subformulas and single negations and let $\pi$ be a dynamic $\Phi$-morphism into the canonical model of some normal logic $\Lambda$ extending ${\bf K4C}^\infty$, where either $\Lambda $ and $\mathfrak S$ are immersive or $\pi$ is distinguished.
Then,
\begin{enumerate}

\item\label{itDD} If $\dd\varphi\in \Phi$ then $\dd\varphi\in \pi(x)$ if and only if there is $y\sqsupset _\gog x$ such that  $\varphi\in \pi(y)$;

\item\label{itBC} If $\varphi$ is a formula (not necessarily in $\Phi$), $i<I$ and $x\in |\gog|$, then $\bc\varphi\in \pi(x)$ if and only $\varphi\in \pi(f_\gog(x))$;

\item\label{itTan} If $\Psi=\{\psi_1,\ldots,\psi_n\}$ then
\begin{itemize}

\item\label{itTanOne} If $\dtan\Psi\in \Phi \cap \pi(x)$ then there are reflexive points $y_1,\ldots,y_n$ such that $x\sqsubseteq_\gog y_1\equiv_\gog y_2 \equiv_\gog\ldots\equiv_\gog y_n$ and $\psi_i \in \pi(y_i)$.

\item\label{itTanTwo} If there are reflexive points $y_1,\ldots,y_n$ such that $x\sqsubseteq_\gog y_1\equiv_\gog y_2 \equiv_\gog\ldots\equiv_\gog y_n$ and $\psi_i \in \pi(y_i)$, then $\dtan\Psi\in \pi(x)$; note that $\dtan\Psi\in \Phi$ is not required.
\end{itemize}
\end{enumerate}
\end{lemma}

\begin{proof}\
\noindent \eqref{itDD} If $\dd\varphi\notin\pi(x)$ then $ \dn \neg \varphi \in\pi(x)$.
Then, if $y\sqsupset_\mathfrak S x$, it follows that $\pi(y)\sqsupseteq_\mathrm c \pi(x)$, so $\neg \varphi \in \pi(y)$ by definition of $\sqsupseteq_\mathrm c$.

If $\dd\varphi\in\pi(x)$, then there exists $v\in W_{\rm c}$ such that $\pi(x) \sqsubset  _{\Phi}  v$ and $\varphi\in v$. By the definition of a dynamic $\Phi$-morphism there exists $v'\in|{\gog}|$ such that $x \sqsubset_\gog v'$ and $v= \pi(v')$, as needed.


\noindent \eqref{itBC} If $\bc\varphi\in\pi(x)$ then $\varphi\in g_{\rm c}(\pi(x))$. By the definition of a dynamic $\Phi$-morphism we get $g_{\rm c}(\pi(x)) = \pi(f_\gog (x))$, and by definition of $g_{\rm c}$, $\varphi\in g_{\rm c}(\pi(x))$, as needed.
The converse implication is obtained by observing that $\neg\bc\varphi$ is equivalent to $\bc\neg\varphi$ and applying the same argument.

\noindent \eqref{itTan} If $\dtan\Psi\in \pi(x) $ then Lemma \ref{finalset} yields the needed witnesses for $x\in \|\dtan\Psi\|_\gog $.
For the converse, we will consider the case where $\Lambda\not\vdash{\rm Cont}_\dd$.
Recall that $x\in |\gog_i|$ for some $i$; we proceed by induction on $i$ to show that if there are $y_1,\ldots,y_n$ such that $x\sqsubseteq_\gog y_1\equiv_\gog y_2 \equiv_\gog\ldots\equiv_\gog y_n$ and $\psi_j \in \pi(y_j)$, then $\dtan\Psi \in \pi(x)$.

Let $y=y_1$.
Since $y\sqsupseteq_\gog x$, it suffices to show that $\dtan\Psi\in \pi(y)$.
Consider the following cases.
If $y$ is $\theta$-final for some $\theta$, observe that every $z\equiv_\gog y$ is $\diamonddot\theta$-final.
It is readily verified that $\boxdot(\diamonddot\theta\to \bigwedge_{\psi\in \Phi} \dd(\psi\wedge\diamonddot\theta)) \in \pi(x)$, hence $\dd \theta\to \dtan\Psi\in \pi(x)$ by Axiom ${\rm Ind}_{\ctan}$, and thus by Modus Ponens, $\dtan\Psi\in \pi(x) $.
If $y$ is reflexive and there is $z$ such that $y=f_\gog(z)$, then by cluster-preservation $C_\gog(y) = f_\gog[C_\gog(z)]$.
Since $\gog$ is distinguished, there is $y'\neq y''$ in $C_\gog(y)$ such that $\pi(y')\neq \pi(y'') $, hence there is a formula $\chi$ such that $\chi\in \pi(y')$ and $\neg\chi\in\pi(y'')$.
We may set $\psi_{n+1}\equiv\chi$, $\psi_{n+2}\equiv\neg\chi$, $y_{n+1} :=y'$ and $y_{n+2} :=y''$.
Letting $z_j \in C_\gog(z)$ be so that $y_j = f_\gog(z_j)$, we readily observe that $z_1\equiv_\gog z_2\equiv_\gog\ldots\equiv_\gog z_{n+2}$ and $\bc\psi_j\in \pi(z_j)$.
Note that $z\in |\gog_{i-1}|$, so we may apply the induction hypothesis to obtain $\dtan\bc(\Psi\cup\{\chi,\neg\chi\})\in \pi(z)$.
By Axiom ${\rm CTan}_\diamonddot$, $\bc\ctan(\Psi\cup\{\chi,\neg\chi\})\in \pi(z)$, hence $\ctan(\Psi\cup\{\chi,\neg\chi\})\in \pi(y)$.
But Lemma~\ref{lemmTanProp} yields $\dtan \Psi \in \pi(y)$, as needed.
Finally, the case where $y $ is irreflexive is impossible by assumption.

The case where $\Lambda$ is immersive is similar, but simplified since we do not need the detour through $\ctan$.
\end{proof}

From this, we easily obtain the following truth preservation lemma.

\begin{lemma}[truth preservation]
Let $\mathfrak S$ be a story of duration $I$ and let $x \in |\mathfrak S_0|$.
Let $\pi$ be a dynamic $\Phi$-morphism into the canonical model of some normal logic $\Lambda$ extending $\mathbf{K4C^\infty}$.
Suppose that $\varphi \in \Phi$ is a formula of $\bc$-depth at most $I$, and either $\Lambda$ and $\gog$ are immersive or $\pi$ is distinguished. Then, $\varphi\in\pi(x)$ iff $x\in \| \varphi\|_\mathfrak S$.
\end{lemma}

\begin{proof}
We must prove the more general claim that if $\varphi$ has $\bc$-depth at most $I-j$ and $x\in |\mathfrak S_j|$, then $\varphi\in\pi(x)$ iff $x\in \| \varphi\|_\mathfrak S$.
The proof proceeds by induction on the complexity of $\varphi$, with each step being immediate from Lemma \ref{lemWit}.
\end{proof}

We will next demonstrate that for every point $w$ in the canonical model, there exists a suitable moment $\mathfrak m$ and a  $\Phi$-morphism mapping $\mathfrak m$ to $w$. 
In order to do this, we define a procedure for constructing new moments from smaller ones.

\begin{definition}[moment construction]
Fix $\Lambda\in \{\mathbf{K4} ,\mathbf{K4D}\}$. Let $\{x\}\subseteq {C}' \subseteq C(x )$ for some $x  $ in the canonical model $\mathfrak{M}^\Lambda_\mathrm c$ with $C'$ finite, and ${\rm q} \in \{{\rm i},{\rm r}\}$ (for `irreflexive' or `reflexive').
Let $\vec{\mathfrak{a}}=\lb \mathfrak{a}_m\rb_{m<N}$ be a sequence of moments.
We define a structure $\noment =  {{\vec{\mathfrak{a}}}\choose {C'}}^{\rm q}_x$, where $\noment=\langle |\noment|,\sqsubset  _\noment,\nu_\noment, r_\noment\rangle$ as follows:
\begin{enumerate}
\item  $|\noment| = C'  \sqcup \bigsqcup_{m<N}| \mathfrak{a}_m | ;$

\item $y\sqsubset  _\noment z $ if either
\begin{itemize}
\item $y,z \in C' $ and ${\rm q} ={\rm r}$,

\item  $y \in C' $ and $z \in |\mathfrak{a}_m| $ for some $m$, or

\item $y,z\in|\mathfrak{a}_m|$ and $y\sqsubset  _{\mathfrak{a}_m} z$ for some $m$;
\end{itemize}
\item $\nu_\noment (p)=\{x\in C':  x\in \nu_{\rm c}(p)\} 
\sqcup\bigsqcup_{m<N}
\nu_{\mathfrak{a}_m}(p)$;

\item $r_\noment = x$.
\end{enumerate}
\end{definition}

The moment construction for some logic $\Lambda$ can be used to produce $\Lambda$-moments.

\begin{lemma}\label{story}
Let $ {C}' \subseteq C(x )$ for some $x $ in the canonical model $\mathfrak{M}^\Lambda_\mathrm c$ and let ${\rm q} \in \{{\rm i},{\rm r}\}$, where either ${\rm q} = {\rm r}$ or $C'$ is a singleton.
Let $\vec{\mathfrak{a}}=\lb \mathfrak{a}_i\rb_{i<N}$ be a sequence of $\Lambda$-moments.
Then, $\moment:={\vec{\mathfrak{a}}\choose C}^{\rm q}_x$ is a $\bf K4$-moment, and if each $\mathfrak{a}_i$ is serial and either $N>0$ or $C$ is reflexive, then it is a $\bf K4D$-moment.
\end{lemma}

\begin{proof}
The relation $\sqsubset _\moment $ is easily seen to be transitive since each $\mathfrak{a}_i$ is weakly transitive and the root sees all other points.
Similarly, any point of $\mathfrak a_i$ has a successor if $\mathfrak a_i$ is serial, and a root point will have a successor in any $\mathfrak a_i$ or be its own successor if $C$ is reflexive.
 \end{proof}

 We use this moment construction to show that every point in the canonical model is the $\Phi$-morphic image of some moment. 
 Below, we denote by $C_\Phi(w)$ the $\sqsubset  _{\Phi}$-cluster of $w$, i.e.\
$$C_\Phi(w)=\{w\}\cup \{v:w\sqsubset  _{\Phi}v \sqsubset  _{\Phi} w\}.$$

\begin{lemma}\label{lemmExistsMoment}
Fix $\Lambda \in \{\mathbf{K4 },\mathbf{K4D} \}$. 
Given a finite set of formulas $\Phi$ with $\dd\top\in \Phi$, for all $w\in W_\mathrm c$ there exists a $\Lambda$-moment $\moment$ and a distinguished $\Phi$-morphism $\pi\colon |\moment|\to W_\mathrm c$ such that $\pi(r_\moment) = w$.
\end{lemma}

\begin{proof}
We prove the stronger claim, that there is a moment $ \moment$ and a map $ \pi\colon | \moment| \to W_{\rm c} $ that is a $p$-morphism on the structure $(W_{\rm c},\sqsubset_\Phi)$ (we will say that $\pi$ is a {\em $p$-morphism with respect to $\sqsubset_\Phi$}).
Let $\sqsubset  _{\Phi}^1$ be the strict $\sqsubset  _{\Phi}$ successor, i.e.\ $w\sqsubset  _{\Phi}^1v$ iff $w \sqsubset  _{\Phi} v$ and $\neg (v \sqsubset  _{\Phi} w)$.
Since $\sqsubset  _{\Phi}^1$ is converse well-founded, we can assume inductively that for each $v$ such that $w \sqsubset  _{\Phi}^1 v$, there is a moment ${\moment_{v}}$ and a $p$-morphism $\pi_{v}:|{\moment_{v}}|\rightarrow W_{\rm c}$ with respect to $\sqsubset  _{\Phi}$ that maps the root of ${\moment_{v}}$ to $v$.
Accordingly, we define a moment
 $$ \moment = {\{{\moment}_v: v \sqsupset  _{\Phi}^1 w \}\choose C_\Phi(w) }^{\rm q}_w,$$
where ${\rm q} = \rm i$ unless $w$ is reflexive and $\Phi$-final or $|C_\Phi(w)|>1$, in which case ${\rm q}={\rm r}$.
Moreover, $C_\Phi(w)$ is finite, so $\moment$ is a $\bf K4$-moment by Lemma \ref{story}.

Next we define a map $ \pi:|\moment |\rightarrow W$ as
$$   \pi(x)=\begin{cases} x &\text{if }x \in C_\Phi(w), \\
\pi_v (x)&\text{if }x\in|\moment _v|. \end{cases}$$

We prove that $  \pi$ is a $p$-morphism for $\sqsubset_\Phi$.
First assume that $  \pi(x)\sqsubset  _{\Phi} v$; we must show that there is $y\sqsupset_\moment x$ so that $ \pi(y) = v $.
Either $x\in C_\Phi(w)$ or $x\in |\moment_u|$ for some $u$.
In the first case, we consider two sub-cases.
If $v\in C_\Phi(w)$ as well, then $ \pi(v) = v$ and we may set $y:=v$.
If not, $   \pi(x)\sqsubset  _{\Phi}^1 v$, and by letting $y$ be the root of $\moment_v$, we see that $ \pi(y) = v$.
If instead $x\in |\moment_u|$ for some $u$, then by assumption $\pi_u$ is a $p$-morphism, immediately yielding the desired $y$ (also satisfying $y \in |\moment_u|$).

Now, suppose $x \sqsubset  _\moment y$. We check that $  \pi(x)\sqsubset  _\Phi   \pi(y)$.
There are two cases to consider:
First suppose that $x\in | \moment_{u}|$ for some $u$ and $y\in |\moment_{v}|$ for some $v$. Then by the definition of the moment construction operation, $ u = v $. By the induction hypothesis, since $ \pi_{u}$ is a $p$-morphism, then $  \pi(x)\sqsubset  _\Phi  \pi(y)$.
Otherwise, suppose that $x\in C_\Phi(w)$.
If $y\in C_\Phi(w)$ we immediately obtain $\pi(x)\sqsubset  _\Phi  \pi(y)$, since $C_\Phi(w)$ is precisely the $\sqsubset  _\Phi$-cluster of $w$ (with the same accessibility relation).
Otherwise, $y\in |\moment_{v}|$ for some $v$, and we let $y'$ be the root of $\moment_v$, so that $ \pi(y') = v$.
Then,  $  \pi(x) \sqsubset  _{\Phi}  \pi(y')$.
If $y=y'$ we are done.
If $y\neq y'$, then since by the induction hypothesis $  \pi_{v}$ is a $p$-morphism with respect to $\sqsubset_\Phi$, then $y'  \sqsubset   y$ implies $  \pi(y') \sqsubset  _ \Phi \pi(y)$. By the  transitivity of $\sqsubset  _\Phi$ we have that $ \pi(x) \sqsubset  _\Phi  \pi(y)$.

Now, for $\Lambda $ extending $ \bf K4D$, observe that each $v$ in the canonical model has a successor satisfying $\top$.
Since $\dd\top\in \Phi$ by assumption and $\dd \top\in v$ by $\rm D$, $v$ has a $\sqsubset_\Phi$-successor satisfying $\top$, so $\sqsubset_\Phi$ is serial.
Since $\pi$ is a $p$-morphism with respect to $\sqsubset_\Phi$, it follows that $\sqsubset_\moment$ is serial as well.

Finally, it is easily checked that $\pi$ is distinguished, using the hypothesis that each $\pi_v$ was distinguished and our definition of $\pi$ and $\rm q$.
\end{proof}

Non-strict monotonicity will create a small technical problem in our proofs.
We wish to construct stories via a step-by-step method.
Suppose that $x\sqsubset x'$ have already been added to our model but $f(x)$ and $f(x')$ have not been defined.
We wish to add points $y $ and $y'$ so that we may set $f(x)=y$ and $f(x' ) = y'$.
However, with non-strict monotonicity alone, we cannot immediately guarantee that $\pi(y)\sqsubset \pi(y')$, as monotonicity only yields $\pi(y)\sqsubseteq \pi(y')$.
In the case that $\pi(y) = \pi(y')$, we may identify $y$ and $y'$ to deal with this issue.
Since we construct our frames top-down, in this case, $y$ and $y'$ will be {\em at the bottom.}

\begin{definition}[pre-$\Phi$-morphism, bottom]
Let $  \moment$ be a moment and $ \pi :  | \moment|\rightarrow W_c$.
We say that $x\in |\moment|$ is \emph{at the bottom} if  $\pi (x)$ is irreflexive and $\pi (x)=\pi (r_\moment)$.

We say that $ \pi:|\moment|\rightarrow W_c$ is a \emph{pre-$\Phi$-morphism} if it fulfils conditions \ref{itAtoms} and \ref{itBack} of a $\Phi$-morphism (Definition \ref{phimorstatic}), and
$ x\sqsubset  _\mathfrak m y $ implies that either $ \pi (x) \sqsubset  _{\rm c}  \pi(y) $ or $ x,y $ are at the bottom.
\end{definition}

\begin{lemma}\label{lemmPremor}
Fix $\Lambda\in\{\mathbf{K4},\mathbf{K4D}\}$. Let $\moment,\hat\noment$ be $\Lambda$-moments such that there is a monotonic, root-preserving, almost injective map $\hat f\colon |\moment| \to |\hat \moment|$ and a pre-$\Phi$-morphism $\hat \pi \colon |\hat \noment| \to W_{\rm c}$.
Then, there exist a $\Lambda$-moment $\noment$, a monotonic, root-preserving, almost injective map $f\colon |\moment|\to|\noment|$, and a $\Phi$-morphism $\pi\colon |\noment|\to W_{\rm c}$.
\end{lemma}

\begin{proof}
This is proven in \cite{CSLpaper}.
One basically takes a quotient, where the points at the bottom are identified.
The only difference is that we are using what in the other paper is the $\bf GL$-bottom, which only identifies copies of the root when it is irreflexive.
In this way, we ensure that $f$ remains injective on the reflexive points.
\end{proof}

Using this, we can prove the existence of an appropriate story that maps to the canonical model.
This is based on the following useful lemma.

\begin{lemma}\label{lemmExistsSuper}
\sloppy Fix \[\Lambda\in\{\mathbf{K4C}^\infty,\mathbf{K4DC}^\infty,\mathbf{K4I}^\infty,\mathbf{K4DI}^\infty\}.\]
Let ${\mathfrak M^\Lambda_{\rm c} = \langle W,\sqsubset,g,\nu\rangle}$ and let $\moment$ be a $\Lambda$-moment. 
Suppose that there exists a $\Phi$-morphism $\pi\colon |\moment|\to W$. 
Then, there exist a moment $ \noment$, a monotonic map $f \colon |\moment| \to | \noment |$, and a $\Phi$-morphism $\rho\colon|\noment | \to W $ such that
$g\circ \pi = \rho \circ f$, where $f$ is immersive whenever $\Lambda$ is immersive.
\end{lemma}

\begin{proof}
We proceed by induction on the height of $\moment$.
Let $C$ be the cluster of $r_\moment$ and let $\vec {\mathfrak a} = \langle\mathfrak a_n\rangle_{n<N}$ be the generated sub-models of the immediate strict successors of $r_\moment$; note that each $\mathfrak a_n$ is itself a moment of smaller height.
By the induction hypothesis, there exist moments $\langle\mathfrak a'_n\rangle_{n<N}$, root-preserving, monotonic maps $f_n \colon |\mathfrak a_n | \to |\mathfrak a'_n|$, and $\Phi$-morphism $\rho_n \colon |\mathfrak a'_n| \to W$ such that  $g\circ \pi_n = \rho_n \circ f_n$.
Moreover, for each $v\sqsupset  _{\Phi} g(r_\moment) $, by Lemma \ref{lemmExistsMoment} there are $\mathfrak b_v$ and a $\Phi$-morphism $\rho_v \colon |\mathfrak b_v| \to W$ mapping the root of $\mathfrak b_v$ to $v$.
If $\Lambda$ is immersive or $ |g_{\rm c}[\pi[C(r_\moment)]]|>1 $, we set $D= \{(x,g_{\rm c}(\pi (x))) : x\in C(r_\moment)\}$ with $(x,y)$ and ${\rm q}: = {\rm r}$ iff $x$ is reflexive (otherwise ${\rm q} := {\rm i}$), $f_D(x) = (x,g_{\rm c}(\pi (x))$ for $x\in C(r_\moment)$, and $\rho_D (x,y): = y$.
Otherwise, $\Lambda$ is not immersive and $g_{\rm c}[\pi[C(r_\moment)]]$ is a singleton, say $v$, and we set $D=\{v\}$ with ${\rm q} ={\rm i}$, $f_D(x) = v$ for all $x\in C(r_\moment)$, and $\rho_D (v): = v$.
This way of defining $D$ will ensure that our map is cluster-preserving and almost injective.

Let $\hat \noment = {\vec{\mathfrak a} *\vec{\mathfrak b}\choose D}^{\rm q}_{g_{\rm c}(w)}$ and define
\[
\hat \rho(w)
=
\begin{cases}
\rho_D(w)  &\text{if $w \in D$},\\
\rho_n (w) & \text{if $w \in |\mathfrak a'_n|$},\\
\rho_v (w) & \text{if $w \in |\mathfrak b_v|$}.
\end{cases}\]
It is not hard to check that $\hat\rho$ is a pre-$\Phi$-morphism from $\hat\noment$ to $W$.
Then, set
\[ 
\hat f(w)
=
\begin{cases}
f_D &\text{if $w \in C(r_\moment)$},\\
f_n(w) & \text{if $w \in |\mathfrak a_n|$}.
\end{cases}
\]
it is easy to see that $\hat f\colon |\moment| \to |\hat\noment|$ is monotonic and satisfies $g \circ  \pi  = \hat \rho  \circ \hat f$.
If each $f_n$ is immersive as is $\Lambda$, then $\hat f$ is in fact strictly monotonic and injective and we are done.

Otherwise, setting $\noment =\quot {\hat\noment}{\hat \rho}$, $f=\quotpi {\hat f}$ and $\rho =\quotpi {\hat \rho}$, Lemma \ref{lemmPremor} implies that $\noment$, $f$ and $\rho$ have the desired properties.
\end{proof}

\begin{proposition}\label{propExistsStory}
\sloppy Fix  \[\Lambda\in\{\mathbf{K4C}^\infty,\mathbf{K4DC}^\infty,\mathbf{K4I}^\infty,\mathbf{K4DI}^\infty\}.\]
Given $I<\omega$ and $w\in W_{\rm c}$, there is a story $\gog$ of duration $I$ and a dynamic $\Phi$-morphism $\pi\colon |\gog| \to W_{\rm c}$ with $w =\pi(r_\gog)$, such that either $\Lambda$ and $\gog$ are immersive or $\pi$ is distinguished.
\end{proposition}

\begin{proof}
We proceed by induction on $I$.
For $I=0$, this is essentially Lemma \ref{lemmExistsMoment}.
Otherwise, by the induction hypothesis, assume that a story $\hat{\gog}$ of depth $I$ and a dynamic $p$-morphism $\hat \pi$ exist. 
By Lemma \ref{lemmExistsSuper}, there is a moment $\gog_{I+1}$, suitable map $f_I\colon |\gog_I| \to |\gog_{I+1}|$, and a $\Phi$-morphism $\pi_{I+1} \colon |\gog_{I+1}| \to W_{\rm c}$ commuting with $f_I$, where $\pi_{I+1}$ can be taken to be distinguished if $\Lambda$ is not immersive.
We define $\gog$ by adding
$\gog_{I+1}$ to $\hat{\gog}$ in order to obtain the desired story.
\end{proof}

 It follows that every satisfiable formula is also satisfiable on a finite story, hence satisfiable on a finite model, yielding the main result of this section.

\begin{theorem}\label{wk4c}
The logics $\mathbf{K4C}^\infty$, $\mathbf{K4DC}^\infty$, $\mathbf{K4I}^\infty$ and $\mathbf{K4DI}^\infty$ are sound and complete for their respective classes of finite dynamic $\Lambda$-frames.
\end{theorem}

\begin{proof}
Soundness follows from Lemma \ref{lemCH} and the well-known soundness results reviewed in Section \ref{secDTL}.
Let $\Lambda\in\{\mathbf{K4C}^\infty , \mathbf{K4DC}^\infty,\mathbf{K4I}^\infty,\mathbf{K4DI}^\infty\}$ and suppose $\Lambda\not\vdash\varphi$. Then in the canonical model $\mathfrak{M}^\Lambda_{\rm c}=\langle W,\sqsubset,g,\nu\rangle $,
there is $w\in W$ that refutes $\varphi$. 
By Proposition~\ref{propExistsStory}, there is a story $\gog$ and a dynamic $\Phi$-morphism $\pi:|\gog|\rightarrow W$ such that $w = \pi(r)$, where $r$ is a root of $\gog$. It follows that $\gog,r \not\models\varphi$. Recall that $\gog$ is a finite dynamic derivative frame, and it is serial if $\Lambda$ contains $\rm D$ and immersive if it contains ${\rm C}_\dd$, as required.
\end{proof}

\section{Metric $d$-completeness}\label{topologicaldcompleteness}

In this section we establish completeness results for dynamic metric systems.
It will be convenient for our Kripke models to duplicate all reflexive points, so that the Cantor derivative can truly be evaluated using points different from the evaluation point.
We use the following well-known construction~\cite{EsakiaAlgebra}.

\begin{definition}
Let $\mathfrak{F}=\langle W,\sqsubset   ,g\rangle$ be a dynamic Kripke frame and let $W^\mathrm i$ and $W^\mathrm r$ be the sets of irreflexive and reflexive points, respectively. 
We define a new frame $\mathfrak{F}_\oplus=\langle W_\oplus,\sqsubset  _\oplus ,g_\oplus\rangle$, where
\begin{enumerate}

\item $W_\oplus = (W^{\rm i}\times \{0\})\cup (W^{\rm r}\times \{0,1\})$;

\item $(w,i) \sqsubset   _\oplus (v,j)$ iff $w\sqsubset   v$;

\item $g_\oplus(w,i) = (g(w),i)$.

\end{enumerate}
\end{definition}

The following is standard and easily verified.

\begin{proposition}\label{propOplus}
If $\mathfrak{F}=\langle W,\sqsubset   ,g\rangle$ is any dynamic derivative frame, then $\mathfrak {F}_\oplus$ is a dynamic derivative frame where reflexive clusters have at least two points and $\pi\colon W_\oplus \to W $ given by $\pi(w,i)=w$ is a surjective, dynamic $p$-morphism.
Moreover, if $\mathfrak F$ is a story, so is $\mathfrak {F}_\oplus$, and if $\mathfrak F$ is immersive, so is $\mathfrak {F}_\oplus$.
\end{proposition}

Next, following Kremer and Mints~\cite{KremerM07}, we need to add limits to our Kripke models.

\begin{definition}
Let $\mathfrak F = \langle W,\sqsubset   ,g\rangle$ be a finite dynamic Kripke frame.
A {\em path} through $\mathfrak F$ is an infinite sequence $\vec w = (w_i)_{i=0}^\infty$ such that $w_i\sqsubseteq w_{i+1}$.
A {\em finite path} is defined similarly but has finitely many elements.
The set of (infinite) paths is denoted $\vec W $.

For a path $\vec w = (w_i)_{i=0}^\infty$, we define $\vec g(\vec w) = (g(w_i))_{i=0}^\infty$.
A {\em limit assignment} is a function $\lim$ assigning to each $\vec w\in\vec W $ an element $\lim \vec w\in W$ such that $\lim \vec w$ occurs infinitely often in $\vec w$, and such that $\lim \vec g(\vec w) = \vec g(\lim \vec w)$.
\end{definition}

\begin{lemma}
Every story has a limit assignment.
\end{lemma}
\begin{proof}
Let $\gog= \langle|\gog|,\sqsubset   ,f ,\nu ,r  \rangle$ be a story of duration $I$.
We will assign limits by linearly ordering the elements of each $|\gog_i|$ and letting $\lim\vec w$ be the least element (under this order) that occurs infinitely often.
In order to define this linear order, it suffices to choose an injective function $h_i\colon |\gog_i| \to \mathbb N$.
For $i=0$ this function may be chosen arbitrarily, but for $i+1$ we must ensure that $h_{i+1}$ is chosen in a way that limits commute.
To this end, first define for $w$ in the range of $f $ $h_{i+1}(w) = \min \{h_i(v):f (v) = w\}$.
Then, if $w$ is not in this range we choose $h_{i+1}(w)$ arbitrarily, provided it is larger than all previously defined values of $h_{i+1}$.

We must check that $\vec f$ commutes with limits.
Let $\vec w $ be a path through $\gog_i$ and $w=\lim\vec w$.
Then, $w$ occurs infinitely often in $\vec w$, so $f(\vec w)$ occurs infinitely often in $\vec f(\vec w)$.
We must show that any $v$ occurring infinitely often in $\vec f(\vec w)$ is so that $h_{i+1} (v) \geq h_{i+1} (f(w))$.
If $v=f(\vec w)$ there is nothing to prove, so we assume otherwise.
But then $f(\vec w)$ and $v$ are in the same cluster since they both occur infinitely often on a path, which means that they are both reflexive; since $f$ is almost injective and cluster-preserving, this means that there is a unique $w'\in |\gog_i|$ such that $f(w') = v$ and $w'$ is in the same cluster as $w$.
But then, $w'$ occurs infinitely often on $\vec w$ (since $f(w')$ occurs infinitely often in $\vec f(\vec w)$), which from $w'\neq\lim\vec w$ implies that $h_i(w') > h_i(w)$, thus $h_{i+1}(v) >h_{i+1} (f(w))$ and $\lim \vec f(\vec w) \neq v$, which by elimination yields $\lim \vec f(\vec w) = f(w)$.
\end{proof}

\begin{definition}
Let $\mathfrak F = \langle W,\sqsubset   ,g\rangle$ be a finite dynamic Kripke frame with a limit assignment.
We define a metric $\delta$ on $\vec W $ such that $\delta(\vec w,\vec v) = 0$ if $\vec w = \vec v$, and otherwise $\delta(\vec w,\vec v) = 2^{-n}$ for the least $n$ such that $w_n\neq v_n$.
\end{definition}

As usual, we write $B_\varepsilon(\vec w)$ for the set $\{\vec v\in \vec W: \delta(\vec w,\vec v)< \varepsilon\}$; this is the {\em ball of radius $\varepsilon$ around $\vec w$.}

\begin{proposition}\label{propCantorMorphism}
Let $\mathfrak F = \langle W,\sqsubset   ,g\rangle$ be a finite story such that every reflexive cluster has at least two elements, and let $\lim$ be a limit assignment on $\mathfrak F $.
The structure $\vec{\mathfrak F} = (\vec W ,\delta,\vec g)$ is a dynamic metric system, and $\lim\colon \vec W \to W$ is a dynamic $p$-morphism.
Moreover, if $g$ is immersive, then so is $\vec g$.
\end{proposition}

\begin{proof}
To check that $\vec g$ is continuous, it suffices to note that if $\delta (\vec w,\vec v) <2^{-n}$, then they coincide on the first $n$ elements, hence so do $\vec g(\vec w),\vec g(\vec v)$ and $\delta(\vec g(\vec w),\vec g(\vec v)) < 2^{-n}$ as well.
It is moreover immediate that ${\lim}   \circ \vec g = g\circ \lim   $ since limits are assumed to have this property.
Now, if $g$ is immersive, we must check that $\vec g$ is locally injective to see that it is also immersive.
Let $\vec w$ be any path through $\gog_i$; then, $B_1(\vec w)$ is a neighbhorhood of $\vec w$ which has the property that any $\vec v \in B_1(\vec w)$ is also a path through $\gog_i$ (since $v_0=w_0$).
But by definition $g$ is injective on $\gog_i$, from which it is readily checked that $\vec g(\vec w) = \vec g(\vec v)$ implies $\vec w=\vec v$, as each of their components must be equal.

It remains to check that $\lim $ is a $p$-morphism.
First we show that if $\vec w \in \vec W$, then there is $\varepsilon >0$ such that  $\vec g[B_\varepsilon (\vec w) ] \subseteq {\uparrow} \lim w$.
Since $\lim\vec w$ occurs infinitely often in $\vec w$, we may choose $n$ so that $w_n =  \lim\vec w$, then define $\varepsilon  = 2^{-n}$.
Let $\vec v \in B_\varepsilon (\vec w) \setminus \{\vec w\} $.
By transitivity, $w_n\sqsubseteq v_m$ for all $m\geq n$, and in particular $\lim\vec w\sqsubseteq \lim\vec v$.
If $\lim w \neq \lim v$ we are done, otherwise since $\vec v\neq\vec w$ by assumption, we can choose $m>n$ such that $v_m\neq w_m$.
This means that at least one of the two is not equal to $\lim \vec w$, say $v_m$.
Since $\lim \vec w$ occurs infinitely often in $\vec v$, we see that $\lim \vec w = v_n\sqsubset v_m \sqsubseteq \lim \vec v$, which by transitivity yields $\lim \vec w \sqsubset \lim \vec v$, as needed.

Now suppose that $\vec g(\vec w) = v$ and let $v'\sqsupset v$ and $\varepsilon >0$; we need to find $\vec u \in B_\varepsilon(\vec w)$ such that $\vec u\neq \vec w$ and $\lim \vec u = v'$.
Since $v$ occurs infinitely often in $\vec w$, we can choose $n$ such that $(w_i)_{i<n} $ extends $\vec a$ and $w_n = v$.
Now consider two cases.
If $v'\neq v$, then for the path $\vec u:=(w_0,\ldots,w_n,v',v',\ldots)$ we have that $\delta(\vec u,\vec w)<\varepsilon$, and $\vec u\neq \vec w$ since $v$ does not occur infinitely often.
Moreover, $\lim\vec u= v'$ since this is the only element that occurs infinitely often.
Otherwise, $v' =  v$ and $v$ is reflexive.
But by assumption there is $v''$ in the same cluster, and we consider $\vec u:=(w_0,\ldots,w_n,v'',v,v,\ldots)$, which as before has the desired properties.
\end{proof}

We wish to show that $\vec W $ is in fact homeomorphic to a subset of the Cantor space.
For this we use the following two results.

\begin{theorem}[Brouwer (e.g.\ \cite{mill})]\label{brouwer}
A metric space is a Cantor space if and only if it is non-empty, perfect, compact, and totally disconnected.
\end{theorem}


In particular, it is well known that the set of branches on the infinite binary tree is homeomorphic to the Cantor set.
It is not hard to see that this binary tree is of the form $\vec W$, where $W$ is a two-element cluster.
More generally, $\vec W$ is {\em always} a Cantor set, provided some mild conditions are satisfied.

\begin{lemma}\label{lemmIsCantor}
Let $\langle W,\sqsubset\rangle$ be a non-empty, finite, transitive frame where every reflexive cluster has at least two points.
If $\langle W,\sqsubset\rangle$ is serial, then $\vec W$ is homeomorphic to the Cantor set, and if $\langle W,\sqsubset\rangle $ is transitive but not necessarily serial, then $\vec W$ is homeomorphic to a closed subset of the Cantor set.
\end{lemma}


\begin{proof}
First assume that $\sqsubset$ is serial.
We have that $\vec W$ is non-empty since $W$ is, it is perfect since for every $\vec w$ and $\varepsilon>0$ we can find $\vec v\neq\vec w$ such that $\delta(\vec v,\vec w)<\varepsilon$; just take $n$ so that $2^{-n}<\varepsilon$ and choose $v\sqsupset w_n$ such that $v\neq w_n$, which exists since $\sqsubset$ is serial and by assumption every reflexive cluster has at least two elements.
Then, for $\vec v :=(w_0,\ldots,w_n,v,v,\ldots)$, we readily see that $\delta(\vec v,\vec w)<\varepsilon$ as needed.
$\vec W$ is compact by König's lemma since the basic opens form a finitely branching tree, and it is totally disconnected since if $A\subseteq \vec W$ has at least two elements, let $n$ be such that $\delta(\vec w,\vec v)=2^{-{n}}$.
Choosing $\varepsilon$ so that $2^{-n-1} <\varepsilon< 2^{-n}$, we see that $B_\varepsilon(\vec w)$ and its complement are both open, but one includes $\vec w$ and one $\vec v$, so $\vec W$ is totally disconnected.
Thus Theorem \ref{brouwer} tells us that $\vec W$ is homeomorphic to the Cantor set.

If $\sqsubset$ is not serial consider $W_\infty$ defined by adding a cluster of two reflexive points $\{u_\infty,v_\infty\}$ which are above every element of $W$.
By the previous item, $\vec W_\infty$ is homeomorphic to the Cantor set, and it is easy to see that $\vec W$ is a closed subset of $\vec W_\infty$.
\end{proof}

Completeness for the Cantor space then follows (see Theorem~\ref{theoCQComp} below).
 Finally, we prove completeness for subspaces of the rational numbers, using the following.

\begin{proposition}[Sierpinski \cite{sierpinski}]\label{sierpinski}
Every two perfect countable metric spaces are homeomorphic to each other.
\end{proposition}

For this, we focus our attention on a countable subspace of $\vec W$; namely, those sequences that are {\em eventually constant.}
Given a frame $\langle W,\sqsubset\rangle$, define $\vec W_0$ to be the set of all $\vec w \in \vec W$ such that there is $n\in\mathbb N$ such that $w_n=w_m$ for all $n>m$.
Clearly, $\vec W_0$ is countable if $W$ is finite, and it inherits the metric (which we denote $\delta_0$) from $\vec W$.

\begin{proposition}
Let $\mathfrak F = \langle W,\sqsubset   ,g\rangle$ be a finite story such that every reflexive cluster has at least two elements and let $\lim$ be a limit assignment on $\mathfrak F $.
The structure $\vec{\mathfrak F} = \langle \vec W_0 ,\delta_0,\vec g_0\rangle$ is a dynamic metric system, and $\lim_0\colon \vec W_0 \to W$ is a dynamic $p$-morphism, where the subindex $0$ denotes the restriction to $\vec W_0$.
Moreover, if $g$ is immersive then so is $\vec g_0$.
\end{proposition}

\begin{proof}
Most of the required properties are inherited from $\vec W$ (e.g., continuity of $\vec g_0$), with the exception of the `back' condition.
However, all of the witnesses produced in the proof of Proposition~\ref{propCantorMorphism} were eventually constants, so the same proof works in this context.
\end{proof}

The following is proven similarly to Lemma \ref{lemmIsCantor}, but with $\vec W$ replaced with $\vec W_0$ and by using Theorem~\ref{sierpinski}.

\begin{lemma}\label{lemmIsQ}
Let $(W,\sqsubset)$ be a non-empty, finite, transitive frame where every reflexive cluster has at least two points.
If $(W,\sqsubset)$ is serial, then $\vec W_0$ is homeomorphic to $\mathbb Q$, and if $(W,\sqsubset)$ is transitive but not necessarily serial, then $\vec W$ is homeomorphic to a subspace of $\mathbb Q$.
\end{lemma}

In conclusion, we obtain the following completeness results.

\begin{theorem}\label{theoCQComp}
Let $X$ be either $\mathbb Q$ or the Cantor set.
\begin{enumerate}
\item ${\bf K4C}^\infty$ is complete for the class of dynamical systems based on a closed subspace of $X$.

\item ${\bf K4I}^\infty$ is complete for the class of injective dynamical systems based on a closed subspace of $X$.

\item ${\bf K4DC}^\infty$ is complete for the class of dynamical systems based on $X$.

\item ${\bf K4DI}^\infty$ is complete for the class of injective dynamical systems based on $X$.

\end{enumerate}
\end{theorem}

\section{Conclusion}

We have developed dynamic topological logics based on the topological $\mu$-calculus, in its tangled presentation, and introduced various axiomatic systems that are sound and complete for their intended interpretations over dynamical systems based on a metric space.
We showed that these completeness results in particular apply to the Cantor space and the rational numbers -- two `canonical' metric spaces.

One may also consider interpretations based on the real line, or on Euclidean spaces in general.
Fern\'andez-Duque~\cite{david2} showed that $\bf S4C$ is complete for the plane, but we cannot expect similar results for ${\bf K4DC}^\infty$, in view of results by Lucero-Bryan and Shehtman~\cite{Lucero-Bryan13,ShehtmanDerived}.
However, it may well be possible to define extensions of ${\bf K4DC}^\infty$ that are complete for Euclidean spaces.

Finally, there is the issue of extending our language to include the `henceforth' operator.
It is our expectation that the $d$-logic of all dynamic metric systems may be axiomatised using the tangled derivative, much as the tangled closure was used to provide an axiomatisation of the closure-based $\mathbf{DTL}$~\cite{david3}.
Fernández-Duque showed how the tangled closure is essential in axiomatising $\mathbf{DTL}$ with the `henceforth' operator, and in future work we plan to show how the same can be done for $\mathbf{DTL}$ with the Cantor derivative.
This follows the work of Fernández-Duque and Montacute who provided a complete axiomatisation for $\mathbf{DTL}$ with the Cantor derivative and `henceforth' for the class of \emph{scattered spaces} \cite{Untangled}.
In this context, we are specifically interested in axiomatising the class of chaotic systems.
We believe that the present work is an important step towards achieving this goal.

\bibliographystyle{plain} 
\bibliography{biblio}

\end{document}